\documentclass[letterpaper,11pt,twoside,keywordsasfootnote,addressatend,noinfoline]{article}

\usepackage{imsart}
\usepackage{fullpage}
\usepackage[english]{babel}
\usepackage{amssymb}
\usepackage{amsmath}
\usepackage{amsthm}
\usepackage{epsfig}
\usepackage{mathrsfs}
\usepackage{color}

\newtheorem{theorem}{Theorem}
\newtheorem{lemma}[theorem]{Lemma}
\newtheorem{proposition}[theorem]{Proposition}

\renewenvironment{proof}{\noindent{\bf Proof.}}{\hspace*{2mm}~$\square$}

\newcommand{\N}{\mathbb{N}}
\newcommand{\Z}{\mathbb{Z}}
\newcommand{\Lat}{\mathscr{L}}
\newcommand{\A}{\mathscr{A}}
\newcommand{\B}{\mathscr{B}}
\newcommand{\C}{\mathscr{C}}

\newcommand{\norm}[1]{|\!|#1|\!|}
\newcommand{\ep}{\epsilon}
\newcommand{\n}{\hspace*{-5pt}}

\DeclareMathOperator{\geometric}{Geometric}
\DeclareMathOperator{\bernoulli}{Bernoulli}
\DeclareMathOperator{\poisson}{Poisson}
\DeclareMathOperator{\exponential}{exponential}
\DeclareMathOperator{\card}{card}

\begin{document}

\begin{frontmatter}
\title{Multitype contact process with sterile states}
\runtitle{Multitype contact process with sterile states}
\author{Nicolas Lanchier, Max Mercer and Hyunsik Yun}
\runauthor{Nicolas Lanchier, Max Mercer and Hyunsik Yun}
\address{School of Mathematical and Statistical Sciences \\ Arizona State University \\ Tempe, AZ 85287, USA. \\ nicolas.lanchier@asu.edu \\ mamerce1@asu.edu \\ hyun26@asu.edu}
\maketitle

\begin{abstract} \
 This paper considers a natural variant of the $d$-dimensional multitype contact process in which individuals can be fertile or sterile.
 Fertile individuals of type $i$ give birth to an offspring of their own type at rate $\lambda_i$, the offspring being fertile with probability $p_i$ and sterile with probability $1 - p_i$, whereas sterile individuals can't give birth.
 Offspring are sent to one of the neighbors of their parent's location and take place in the system if and only if the target site is empty.
 All the individuals die at rate one regardless of their type and regardless of whether they are fertile or sterile.
 Our main results show some qualitative disagreements between the spatial model and its nonspatial mean-field approximation that are more pronounced when the probability $p_i$ is small.
 More precisely, for the mean-field model, in the presence of only one type, survival occurs when $\lambda_i p_i > 1$, and in the presence of two types, the type with the largest $\lambda_i p_i$ wins.
 In contrast, though the analysis of the spatial model shows a similar behavior when $p_i$ is close to one, in the presence of only one type, extinction always occurs when $p_i < 1/4d$.
 Similarly, a type with $\lambda_i > \lambda_c =$ critical value of the contact process and $p_i = 1$ is more competitive than a type with $\lambda_i$ arbitrarily large but $p_i < 1/4d$, showing that the product $\lambda_i p_i$ no longer measures the competitiveness.
 These results underline the effects of space in the form of local interactions.
\end{abstract}



\end{frontmatter}


\section{Introduction}
\label{sec:intro}
 Neuhauser's multitype contact process~\cite{neuhauser_1992} models the competition between two different types of individuals on the~$d$-dimensional integer lattice.
 Each site of the lattice is either in state~0~=~empty, state~1~=~occupied by a type~1 individual, or state~2~=~occupied by a type~2 individual, so the state of the process at time~$t$ is a spatial configuration~$\xi_t : \Z^d \to \{0, 1, 2 \}$.
 Individuals of type~$i$ give birth at rate~$\lambda_i$ and die at rate one.
 Offspring are randomly sent to one of the~$2d$ neighbors of their parent's location, and take place in the system if the target site is empty.
 In particular, the local transition rates of the process at site~$x \in \Z^d$ are given by
 $$ \begin{array}{rclcrcl}
      0 \to 1 & \hbox{at rate} & \lambda_1 f_1 (x, \xi), & \qquad & 1 \to 0 & \hbox{at rate} & 1, \vspace*{4pt} \\
      0 \to 2 & \hbox{at rate} & \lambda_2 f_2 (x, \xi), & \qquad & 2 \to 0 & \hbox{at rate} & 1, \end{array} $$
 where~$f_i (x, \xi)$ denotes the fraction of neighbors of site~$x$ that are in state~$i$.
 In the presence of only one type, the process reduces to Harris' contact process~\cite{harris_1974}.
 For an overview of the contact process and the multitype contact process, as well as variants of interest in epidemiology, population ecology, and community ecology, we refer the reader to~\cite[Chapters~1--4]{lanchier_2024}.
 This paper introduces another variant of the~(multitype) contact process of interest in ecology, in which individuals can now be fertile or sterile.
 Fertile individuals of type~$i$ again give birth at rate~$\lambda_i$, while sterile individuals cannot give birth, and the offspring are independently fertile with probability~$p_i$ and sterile with probability~$q_i = 1 - p_i$.
 Using state~$+i$ for fertile individuals of type~$i$, and state~$-i$ for sterile individuals of type~$i$, the local transition rates become
\begin{equation}
\label{eq:IPS}
\begin{array}{rclcrcl}
  0 \to +1 & \hbox{at rate} & \lambda_1 p_1 f_{+1} (x, \xi), & \qquad & +1 \to 0 & \hbox{at rate} & 1, \vspace*{4pt} \\
  0 \to -1 & \hbox{at rate} & \lambda_1 q_1 f_{+1} (x, \xi), & \qquad & -1 \to 0 & \hbox{at rate} & 1, \vspace*{4pt} \\
  0 \to +2 & \hbox{at rate} & \lambda_2 p_2 f_{+2} (x, \xi), & \qquad & +2 \to 0 & \hbox{at rate} & 1, \vspace*{4pt} \\
  0 \to -2 & \hbox{at rate} & \lambda_2 q_2 f_{+2} (x, \xi), & \qquad & -2 \to 0 & \hbox{at rate} & 1. \end{array}
\end{equation}
 Neuhauser's multitype contact process is the particular case where~$p_1 = p_2 = 1$, while Harris' contact process is obtained by also assuming that only one type is present.
 To simplify the notation when dealing with the process with only one type, we will drop the label referring to the type, and write the states as~$+$ and~$-$ and the parameters of the system as~$\lambda, p, q$. \vspace*{5pt}


\noindent{\bf Mean-field model}.
 To understand the effects of local interactions, we first study the nonspatial deterministic mean-field approximation of the stochastic process.
 Assuming that the population is well-mixing, and letting~$u_{\pm i}$ denote the density of sites in state~$\pm i$, the mean-field model of the process~\eqref{eq:IPS} is given by the system of coupled differential equations
 $$ \begin{array}{rclcrcl}
      u_{+1}' & \n = \n & \lambda_1 p_1 u_{+1} u_0 - u_{+1}, \qquad & u_{-1}' & \n = \n & \lambda_1 q_1 u_{+1} u_0 - u_{-1}, \vspace*{4pt} \\
      u_{+2}' & \n = \n & \lambda_2 p_2 u_{+2} u_0 - u_{+2}, \qquad & u_{-2}' & \n = \n & \lambda_2 q_2 u_{+2} u_0 - u_{-2}, \end{array} $$
 where~$u_0 = 1 - u_{+1} - u_{-1} - u_{+2} - u_{-2}$ is the density of empty sites.
 In the presence of only one type, we will prove that the population survives if and only if~$\lambda p > 1$ in the sense that there is a unique interior fixed point~(that is globally stable) in the two-dimensional simplex given by
\begin{equation}
\label{eq:Q}
  Q = (u_+, u_-) = \bigg(\bigg(1 - \frac{1}{\lambda p} \bigg) p, \bigg(1 - \frac{1}{\lambda p} \bigg) q \bigg)
\end{equation}
 if and only if~$\lambda p > 1$.
 This implies that, in the presence of two types, when~$\lambda_1 p_1, \lambda_2 p_2 > 1$, there are two nontrivial boundary fixed points given by
 $$ \begin{array}{rcl}
      Q_1 = (u_{+1}, u_{-1}, u_{+2}, u_{-2}) & \n = \n &
    \displaystyle \bigg(\bigg(1 - \frac{1}{\lambda_1 p_1} \bigg) p_1, \bigg(1 - \frac{1}{\lambda_1 p_1} \bigg) q_1, 0, 0 \bigg), \vspace*{8pt} \\
      Q_2 = (u_{+1}, u_{-1}, u_{+2}, u_{-2}) & \n = \n &
    \displaystyle \bigg(0, 0, \bigg(1 - \frac{1}{\lambda_2 p_2} \bigg) p_2, \bigg(1 - \frac{1}{\lambda_2 p_2} \bigg) q_2 \bigg). \end{array} $$
 We will prove that type~1 wins if in addition~$\lambda_1 p_1 > \lambda_2 p_2$~(and type~2 wins if the inequality is reversed) in the sense that~$Q_1$ and~$Q_2$ are the only two nontrivial fixed points,~$Q_1$ is locally stable, and~$Q_2$ is unstable.
 In the neutral case~$\lambda_1 p_1 = \lambda_2 p_2 > 1$, coexistence occurs in the sense that the segment line connecting~$Q_1$ and~$Q_2$ is a locally stable collection of fixed points. \vspace*{5pt}


\noindent{\bf Spatial model}.
 Our analysis of the stochastic process shows that the inclusion of space in the form of local interactions strongly affects the qualitative behavior of the system, with more pronounced effects as~$p$ gets smaller.
 In the nonspatial mean-field model, survival of a single type occurs if and only if~$\lambda p > 1$, and the winner in the presence of two types is the type with the highest~$\lambda_i p_i$ product.
 In contrast, in the spatial model, the population survives when~$\lambda p > \lambda_c =$ critical value of the contact process and~$p$ is close to one, but when~$p$ is smaller than a critical value~$p_- \geq 1/4d$, the population dies out even when~$\lambda = \infty$.
 Similarly, a type with~$\lambda > \lambda_c$ and~$p = 1$ wins against a type that has~$p < p_-$ and~$\lambda = \infty$, showing that the type with the highest~$\lambda_i p_i$ is not always the most competitive.
 To give a rigorous statement of our results, starting from a translation-invariant configuration with a positive density of~$+i$s, we say that
 $$ \begin{array}{rll}
    \hbox{type~$i$ survives when} & \lim_{t \to \infty} P (\xi_t (x) = +i) > 0 & \hbox{for all} \ x \in \Z^d, \vspace*{4pt} \\
    \hbox{type~$i$ dies out when} & \lim_{t \to \infty} P (\xi_t (x) = +i) = 0 & \hbox{for all} \ x \in \Z^d. \end{array} $$
 In the presence of two types, we say that type~$i$ wins when type~$i$ survives while the other type dies out, and that both types coexist when they both survive. \\
\indent
 We first study the survival/extinction phase of the single-type process.
 For the basic contact process~(the particular case~$p = 1$), Harris~\cite{harris_1974} proved the existence of a nondegenerate critical value~$\lambda_c$ that depends on the spatial dimension such that the process dies out when~$\lambda < \lambda_c$ but survives when~$\lambda > \lambda_c$.
 Using a block construction supplemented with a continuity argument, Bezuidenhout and Grimmett~\cite{bezuidenhout_grimmett_1990} also proved that extinction occurs at~$\lambda = \lambda_c$.
 For general values of the probability~$p$, because fertile individuals give birth to fertile individuals at rate~$\lambda p$ but also to sterile individuals blocking the spread of the population at rate~$\lambda q$, extinction should occur when~$\lambda p \leq \lambda_c$.
 Using standard coupling techniques to compare our process with the basic contact process with parameter~$\lambda p$ confirms extinction in this case.
\begin{theorem}
\label{th:coupling}
 $\lambda p \leq \lambda_c \,\Rightarrow$ extinction.
\end{theorem}
\noindent
 This result is in agreement with extinction in the mean-field model when~$\lambda p < 1$.
 We now look at the behavior of the system when~$\lambda p > \lambda_c$, starting with the case where~$p$ is close to one.
 Using block constructions, Bezuidenhout and Grimmett~\cite{bezuidenhout_grimmett_1990}, and later Durrett and Neuhauser~\cite{durrett_neuhauser_1997}, proved that the supercritical contact process properly rescaled in space and time dominates oriented site percolation with parameter arbitrarily close to one.
 Once the space and time scales are fixed, large deviation estimates imply that the probability of a birth of a sterile individual in a space-time block can be made arbitrarily small by choosing~$p$ close to one, thus showing that the stochastic domination still holds and extending the survival region to~$p < 1$.
\begin{theorem}
\label{th:survival}
 For all~$\lambda > \lambda_c$, there exists~$p_+ = p_+ (\lambda) < 1$ such that~$p > p_+ \,\Rightarrow$ survival.
\end{theorem}
\noindent
 This result is in agreement with survival in the mean-field model when~$\lambda p > 1$.
\begin{figure}[t]
\centering
\scalebox{0.42}{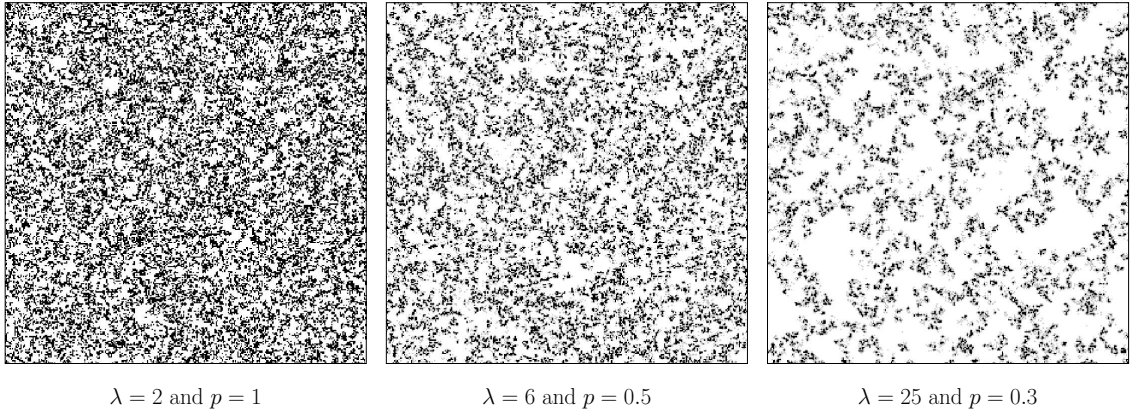}
\vspace*{-5pt}
\caption{\upshape Snapshots at time~1$00$ of the contact process with a sterile state.
                  The black sites represent the fertile individuals, the gray sites the sterile individuals, and the white sites the empty sites.}
\label{fig:sterile}
\end{figure}
 To study the process when~$\lambda p > \lambda_c$ with~$p > 0$ small, notice that the number of deaths in the neighborhood of a fertile individual during its lifespan is dominated by~$N =$ shifted geometric random variable with mean~$2d$, so the number of offspring produced by this individual is dominated by~$2d + N$.
 Since in addition each offspring is independently fertile with probability~$p$, we deduce that the expected number of~$+1$s produced by a single~$+1$ is bounded by
 $$ p E (2d + N) = p (2d + 2d) = 4dp. $$
 In particular, when~$p < 1/4d$, the number of~$+1$s in the process starting from a single~$+1$ is dominated by the total progeny of a subcritical Galton-Watson branching process, which shows extinction.
 To deal with more general initial configurations and study the multitype process later, we prove an exponential decay of the total progeny of the branching process and use a block construction to also show an exponential decay of the spatial model. 
\begin{theorem}
\label{th:extinction}
 There exists~$p_- \geq 1/4d$ such that~$p < p_- \,\Rightarrow$ extinction for all~$\lambda$.
\end{theorem}
\noindent
 This result is in sharp contrast with the mean-field model in which, for all~$p > 0$, the population survives when~$\lambda > 1/p$.
 This qualitative difference is due to the effect of global versus local interactions.
 In the mean-field model, even if~$p > 0$ is very small, provided~$\lambda$ is large, fertile individuals can produce and disperse a large number of individuals, a positive fraction of which are fertile.
 In contrast, in the interacting particle system, because fertile individuals are now blocked by their own offspring, they can only produce a limited number of individuals~(even if~$\lambda = \infty$), all of which are likely to be sterile when~$p > 0$ is small.
 The snapshots of the single-type process in Figure~\ref{fig:sterile} show that, while increasing the product~$\lambda p$ and decreasing~$p$, the density of individuals in the spatial model decreases, which supports/illustrates Theorems~\ref{th:survival}--\ref{th:extinction}, and contrasts with the monotonicity of the density of individuals~\eqref{eq:Q} with respect to~$\lambda p$ in the mean-field model.
\begin{figure}[t!]
\centering
\scalebox{0.40}{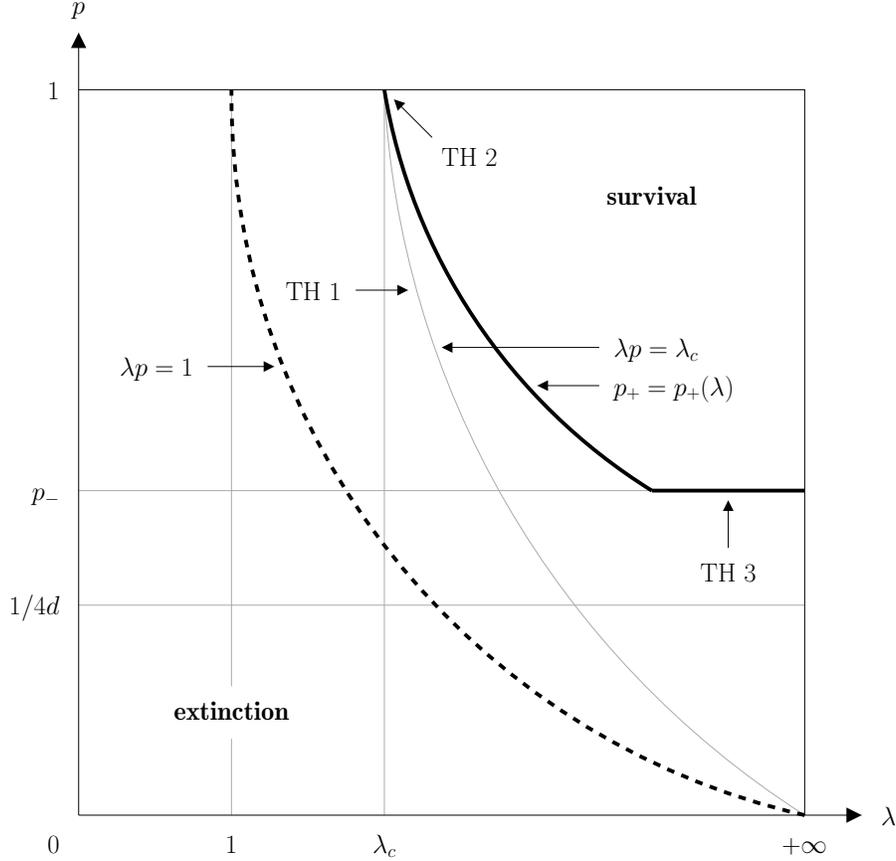}
\caption{\upshape{
 Phase structure of the single-type model.
 The dashed curve with equation~$\lambda p = 1$ separates the extinction phase from the survival phase for the mean-field model.
 The solid black curve shows the phase transition of the spatial model.
 This curve lies above the gray curve~$\lambda p = \lambda_c$ according to Theorem~\ref{th:coupling} and above the gray line~$p = 1/4d$ according to Theorem~\ref{th:extinction}.
 This curve also converges to one as~$\lambda \downarrow \lambda_c$ according to Theorems~\ref{th:survival}--\ref{th:extinction}.}}
\label{fig:phase}
\end{figure}
 Figure~\ref{fig:phase} summarizes our results for the spatial and nonspatial single-type models. \\
\indent 
 For the multitype contact process~(the particular case~$p_1 = p_2 = 1$), Neuhauser~\cite{neuhauser_1992} proved, using duality techniques, that the type with the largest birth rate wins, provided the birth rate is larger than~$\lambda_c$.
 She also proved that, in the neutral case~$\lambda_1 = \lambda_2 > \lambda_c$, the process clusters in dimensions one and two, while coexistence occurs in higher dimensions.
 Returning to the process with fertile and sterile states, Theorems~\ref{th:survival}--\ref{th:extinction} suggest that, when~$\lambda_1 > \lambda_c$ and~$p_1 = 1$, type~1 individuals win whenever~$p_2 < 1/4d$.
 In this case, extinction of the~2s directly follows from Theorem~\ref{th:extinction}.
 Survival of the~1s, however, is unclear, because the~1s are now blocked by the~2s.
 To prove survival, we will study an infinite collection of single-type processes coupled with the multitype process, and use linear growth of the~1s in the absence of the~2s and exponential decay of the~2s in the absence of the~1s to deduce the existence of clusters of~1s that do not interact with the~2s.
\begin{theorem}
\label{th:mcp}
 $\lambda_1 > \lambda_c$ and $p_1 = 1$ and~$p_2 < 1/4d \,\Rightarrow$ type~1 wins (even when~$\lambda_2 = \infty$).
\end{theorem}
\noindent
 This result is again in sharp contrast with the mean-field model.
 Indeed, in the mean-field model, the type with the highest~$\lambda_i p_i$ product wins provided the product is larger than one.
 In contrast, the previous theorem shows that, in the interacting particle system, a type with~$\lambda_i p_i$ barely larger than~$\lambda_c$ can win against a type with~$\lambda_i p_i = \infty$. \\
\indent
 The rest of the paper is devoted to the proof of our results.
 Section~\ref{sec:mf} focuses on the analysis of the mean-field model: identification of all the fixed points along with their stability.
 Section~\ref{sec:coupling} shows how to construct the spatial model graphically from a collection of Poisson processes, which is then used to couple the single-type model and the contact process, and deduce Theorem~\ref{th:coupling}.
 Section~\ref{sec:survival} combines a block construction and a perturbation argument to prove Theorem~\ref{th:survival}.
 Section~\ref{sec:extinction} studies the progeny of the subcritical Galton-Watson branching process above and uses another block construction to deduce Theorem~\ref{th:extinction}.
 Finally, Section~\ref{sec:mcp} relies on coupling techniques and various results collected in the proofs of Theorems~\ref{th:survival}--\ref{th:extinction} to deduce Theorem~\ref{th:mcp}.


\section{Mean-field model}
\label{sec:mf}
 In the presence of only one type, letting~$u_+$ denote the density of fertile individuals, and~$u_-$ the density of sterile individuals, the mean-field model reduces to
\begin{equation}
\label{eq:mf-01}
\begin{array}{rcl}
  u_+' & \n = \n & F_+ (u_+, u_-) = \lambda p u_+ (1 - u_+ - u_-) - u_+, \vspace*{4pt} \\
  u_-' & \n = \n & F_- (u_+, u_-) = \lambda q u_+ (1 - u_+ - u_-) - u_-. \end{array}
\end{equation}
 The extinction state~$u_+ = u_- = 0$ is always a fixed point.
 Assuming that~$u_+ \neq 0$ and setting the right-hand side of the first equation in~\eqref{eq:mf-01} equal to zero, we get
\begin{equation}
\label{eq:mf-02}
\lambda p (1 - u_+ - u_-) = 1 \quad \hbox{so} \quad u_+ + u_- = 1 - 1 / \lambda p.
\end{equation}
 Substituting in the second equation also gives
\begin{equation}
\label{eq:mf-03}
  u_- = \lambda q u_+ (1 - (1 - 1 / \lambda p)) = q u_+ / p = \quad \hbox{so} \quad q u_+ = p u_-.
\end{equation}
 Combining~\eqref{eq:mf-02}--\eqref{eq:mf-03}, we deduce that there exists a unique~(biologically relevant) interior fixed point in the two-dimensional simplex~$\Delta_2$ given by
\begin{equation}
\label{eq:mf-04}
 Q = (1 - 1 / \lambda p) \cdot (p, q) \in \Delta_2 \quad \hbox{if and only if} \quad \lambda p > 1.
\end{equation}
 To study the local stability, notice that the Jacobian matrix of~\eqref{eq:mf-01} can be written as
 $$ \mathcal J = \left(
    \begin{array}{cc} \lambda p (1 - 2 u_+ - u_-) - 1 & - \lambda p u_+ \vspace*{4pt} \\ \lambda q (1 - 2 u_+ - u_-) & - \lambda q u_+ - 1 \end{array} \right). $$
 The trace and determinant of the matrix are given by
 $$ \begin{array}{rcl}
      T (u_+, u_-) & \n = \n & \lambda p (1 - u_+ - u_-) - \lambda u_+ - 2, \vspace*{4pt} \\
      D (u_+, u_-) & \n = \n & \lambda u_+ - \lambda p (1 - u_+ - u_-) + 1 = - T (u_+, u_-) - 1. \end{array} $$
 At the trivial fixed point~$u_+ = u_- = 0$, this reduces to
\begin{equation}
\label{eq:mf-05}
  T (0) = \lambda p - 2 \quad \hbox{and} \quad D (0) = 1 - \lambda p,
\end{equation}
 while at the interior fixed point~\eqref{eq:mf-04}, this reduces to
\begin{equation}
\label{eq:mf-06}
\begin{array}{rcl}
  T (Q) & \n = \n & \lambda p (1 - (1 - 1 / \lambda p)) - \lambda p (1 - 1 / \lambda p) - 2 = - \lambda p, \vspace*{4pt} \\
  D (Q) & \n = \n & - T (Q) - 1 = \lambda p - 1. \end{array}
\end{equation}
 Looking at the sign of the trace and determinant in~\eqref{eq:mf-05}--\eqref{eq:mf-06} to deduce the sign of the eigenvalues shows that the trivial fixed point is locally stable if and only if~$\lambda p < 1$, while the interior fixed point is locally stable if and only if~$\lambda p > 1$.
 To also study the limiting behavior of the system and the global stability of the fixed points, let~$\phi = 1 / (u_+ u_-)$ and notice that
 $$ \begin{array}{rcl}
    \nabla (\phi F_+, \phi F_-) & \n = \n &
    \displaystyle \frac{\partial (\phi F_+)}{\partial u_+} + \frac{\partial (\phi F_-)}{\partial u_-} \vspace*{8pt} \\ & \n = \n &
    \displaystyle \frac{\partial}{\partial u_+} \bigg(\frac{\lambda p (1 - u_+ - u_-) - 1}{u_-} \bigg) +
    \displaystyle \frac{\partial}{\partial u_-} \bigg(\frac{\lambda q (1 - u_+ - u_-)}{u_-} - \frac{1}{u_+} \bigg) \vspace*{8pt} \\ & \n = \n &
    \displaystyle - \frac{\lambda p}{u_-} - \frac{\lambda q (1 - u_+)}{u_-^2} < 0 \end{array} $$
 for all~$(u_+, u_-)$ in the interior of~$\Delta_2$.
 In particular, the Bendixson-Dulac theorem excludes the existence of periodic orbits, from which we conclude that
\begin{itemize}
\item
 the population dies out when~$\lambda p < 1$ in the sense that, regardless of the initial densities of fertile/sterile individuals, we have $u_+ + u_- \to 0$ as time goes to infinity, \vspace*{4pt}
\item
 the population survives when~$\lambda p > 1$ in the sense that the interior fixed point~$Q$ belongs to the interior of the simplex, and~$(u_+, u_-) \to Q$ whenever~$u_+ > 0$ initially.
\end{itemize}
 Figure~\ref{fig:MFt1} shows the solution curves of the mean-field model in the survival phase~$\lambda p > 1$.
\begin{figure}[t!]
\centering
\scalebox{0.42}{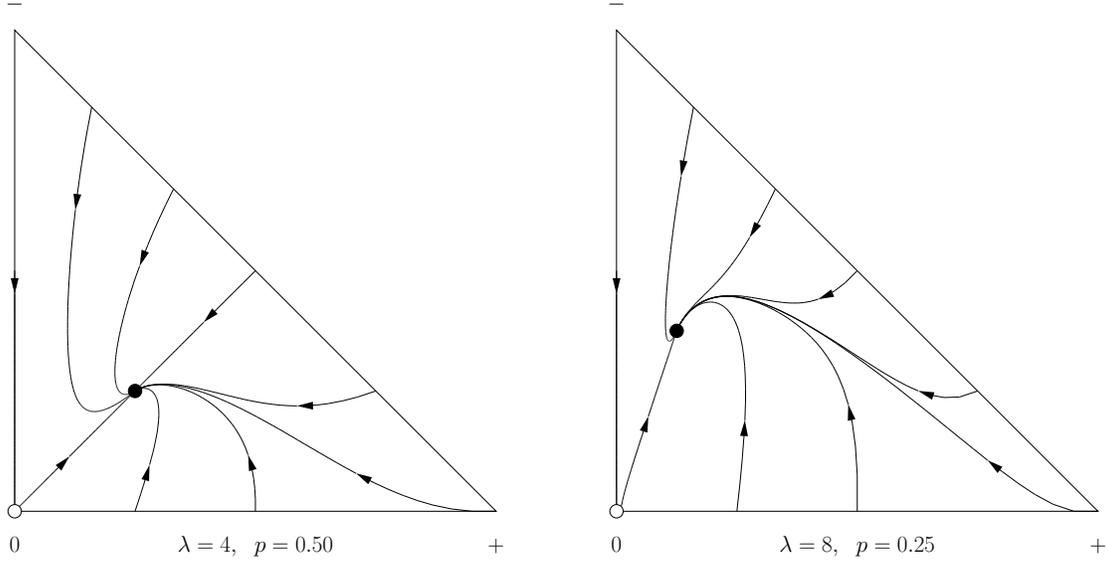}
\caption{\upshape{
 Solution curves of the mean-field model~\eqref{eq:mf-01}.
 The horizontal axis represents the density of fertile individuals, while the vertical axis represents the density of sterile individuals.
 In both pictures~$\lambda p = 2 > 1$ so all the solution curves starting from~$u_+ \neq 0$ converge to the unique interior fixed point~$Q$.
 The limiting density of individuals is the same in both pictures, but the limiting density of fertile individuals is smaller and the limiting density of sterile individuals larger in the second picture because~$p$ is smaller.}}
\label{fig:MFt1}
\end{figure}
 In the presence of two types, letting~$u_{+i}$ denote the density of fertile type~$i$ individuals, and~$u_{-i}$ the density of sterile type~$i$ individuals, the mean-field model becomes
\begin{equation}
\label{eq:mf-07}
\begin{array}{rclcrcl}
  u_{+1}' & \n = \n & \lambda_1 p_1 u_{+1} u_0 - u_{+1}, \qquad & u_{-1}' & \n = \n & \lambda_1 q_1 u_{+1} u_0 - u_{-1}, \vspace*{4pt} \\
  u_{+2}' & \n = \n & \lambda_2 p_2 u_{+2} u_0 - u_{+2}, \qquad & u_{-2}' & \n = \n & \lambda_2 q_2 u_{+2} u_0 - u_{-2}, \end{array}
\end{equation}
 where~$u_0 = 1 - u_{+1} - u_{-1} - u_{+2} - u_{-2}$.
 To simply the notation, let
 $$ \Lambda_1 = 1 - 1 / \lambda_1 p_1 \quad \hbox{and} \quad \Lambda_2 = 1 - 1 / \lambda_2 p_2. $$
 The analysis in the presence of one type shows that, in addition to the trivial fixed point, there are two fixed points on the boundary of the four-dimensional simplex~$\Delta_4$ given by
 $$ \begin{array}{rclcl}
      Q_1 & \n = \n & \Lambda_1 \cdot (p_1, q_1, 0, 0) & \hbox{if and only if} & \lambda_1 p_1 > 1, \vspace*{4pt} \\
      Q_2 & \n = \n & \Lambda_2 \cdot (0, 0, p_2, q_2) & \hbox{if and only if} & \lambda_2 p_2 > 1. \end{array} $$
 To study the existence of interior fixed points, assume that~$u_{\pm 1}, u_{\pm 2} > 0$ and set the right-hand side of the equations on the left of~\eqref{eq:mf-07} equal to zero.
 This gives
 $$ \lambda_1 p_1 u_0 - 1 = \lambda_2 p_2 u_0 - 1 = 0 \quad \hbox{and} \quad \lambda_1 p_1 = \lambda_2 p_2 = 1 / u_0 > 1, $$
 so coexistence is only possible when~$\lambda_1 p_1 = \lambda_2 p_2 > 1$.
 Some basic algebra confirms that, in this case, there is a collection of interior fixed points given by the segment line
 $$ (Q_1, Q_2) = \{P_{\theta} = \theta Q_1 + (1 - \theta) Q_2 : 0 < \theta < 1 \} \subset \Delta_4 \ \ \hbox{connecting~$Q_1$ and~$Q_2$}, $$
 so coexistence is possible.
 Indeed, at point~$P_{\theta}$, when~$\lambda_1 p_1 = \lambda_2 p_2$,
 $$ u_0 = 1 - \theta \Lambda_1 (p_1 + q_1) - (1 - \theta) \Lambda_2 (p_2 + q_2) = 1 - \Lambda_1 = 1 / \lambda_1 p_1, $$
 from which it follows that
 $$ \begin{array}{rcl}
      u_{+1}' & \n = \n & \lambda_1 p_1 u_{+1} u_0 - u_{+1} = (\lambda_1 p_1 u_0 - 1) u_{+1} = (1 - 1) u_{+1} = 0, \vspace*{4pt} \\
      u_{-1}' & \n = \n & \lambda_1 q_1 u_{+1} u_0 - u_{-1} = \lambda_1 q_1 \theta \Lambda_1 p_1 / \lambda_1 p_1 - \theta \Lambda_1 q_1 = \theta \Lambda_1 q_1 - \theta \Lambda_1 q_1 = 0.\end{array} $$
 By symmetry, the same holds for~$u_{\pm 2}'$, and some basic algebra shows that there are no other interior fixed points.
 Now, notice that the Jacobian matrix of~\eqref{eq:mf-07} is given by
\begin{equation}
\label{eq:jacobian}
\mathcal J = \left(\begin{array}{cccc}
   \lambda_1 p_1 (u_0 - u_{+1}) - 1 & - \lambda_1 p_1 u_{+1}     & - \lambda_1 p_1 u_{+1}           & - \lambda_1 p_1 u_{+1} \vspace*{4pt} \\
   \lambda_1 q_1 (u_0 - u_{+1})     & - \lambda_1 q_1 u_{+1} - 1 & - \lambda_1 q_1 u_{+1}           & - \lambda_1 q_1 u_{+1} \vspace*{4pt} \\
 - \lambda_2 p_2 u_{+2}             & - \lambda_2 p_2 u_{+2}     & \lambda_2 p_2 (u_0 - u_{+2}) - 1 & - \lambda_2 p_2 u_{+2} \vspace*{4pt} \\
 - \lambda_2 q_2 u_{+2}             & - \lambda_2 q_2 u_{+2}     & \lambda_2 q_2 (u_0 - u_{+2})     & - \lambda_2 q_2 u_{+2} - 1 \end{array} \right)
\end{equation}
 At point~$Q_1$, using that~$u_0 = 1 / \lambda_1 p_1$, the matrix reduces to
 $$ \begin{array}{rcl}
    \mathcal J (Q_1) = \left(\begin{array}{cc} A & B \vspace*{4pt} \\ 0 & D \end{array} \right) & \hbox{where} & \hspace*{2pt}
                   A = \left(\begin{array}{cc}  p_1 (1 - \lambda_1 p_1)             & p_1 (1 - \lambda_1 p_1) \vspace*{4pt} \\
                                                q_1 / p_1 + q_1 (1 - \lambda_1 p_1) & q_1 (1 - \lambda_1 p_1) - 1 \end{array} \right) \vspace*{12pt} \\ &&
                   D = \left(\begin{array}{cc} \lambda_2 p_2 / \lambda_1 p_1 - 1    & 0 \vspace*{4pt} \\
                                               \lambda_2 q_2 / \lambda_1 p_1        & - 1 \end{array} \right). \end{array} $$
 Letting~$\alpha = 1 - \lambda_1 p_1$ to simplify the notation, we get
 $$ \begin{array}{rcl}
    \det (A - X I_2) & \n = \n & (\alpha p_1 - X)(\alpha q_1 - 1 - X) - \alpha q_1 (1 + \alpha p_1) \vspace*{4pt} \\
                     & \n = \n & X^2 + (1 - \alpha p_1 - \alpha q_1) X + \alpha p_1 (\alpha q_1 - 1) - \alpha q_1 (1 + \alpha p_1) \vspace*{4pt} \\
                     & \n = \n & X^2 + (1 - \alpha) X - \alpha = (X + 1)(X - \alpha).\end{array} $$
 It follows that the spectrum of~$\mathcal J (Q_1)$ is given by
 $$ \sigma (\mathcal J (Q_1)) = \sigma (A) \cup \sigma (D) = \{-1, \alpha = 1 - \lambda_1 p_1, \lambda_2 p_2 / \lambda_1 p_1 - 1 \} $$
 therefore the fixed point~$Q_1$, which is on the boundary of the simplex~$\Delta_4$ if and only if~$\lambda_1 p_1 > 1$, is locally stable if and only if we also have~$\lambda_1 p_1 > \lambda_2 p_2$.
 By symmetry, we have a similar result for the fixed point~$Q_2$.
 To study the local stability of the fixed points~$P_{\theta}$ when~$\lambda_1 p_1 = \lambda_2 p_2 > 1$, notice that, in this case, the matrix~\eqref{eq:jacobian} minus~$X I_4$ reduces to
 $$ \left(\begin{array}{cccc}
    \theta \alpha p_1 - X          & \theta \alpha p_1         & \theta \alpha p_1                   & \theta \alpha p_1 \vspace*{4pt} \\
     q_1 / p_1 + \theta \alpha q_1 & \theta \alpha q_1 - 1 - X & \theta \alpha q_1                   & \theta \alpha q_1 \vspace*{4pt} \\
    (1 - \theta) \alpha p_2        & (1 - \theta) \alpha p_2   & (1 - \theta) \alpha p_2 - X         & (1 - \theta) \alpha p_2 \vspace*{4pt} \\
    (1 - \theta) \alpha q_2        & (1 - \theta) \alpha q_2   & q_2 / p_2 + (1 - \theta) \alpha q_2 & (1 - \theta) \alpha q_2 - 1 - X \end{array} \right) $$
 where~$\alpha = 1 - \lambda_1 p_1 = 1 - \lambda_2 p_2$.
 To compute the determinant, we first subtract the second column from the other three columns, then expand along the last column to get
 $$ \begin{array}{rcl}
    \det (\mathcal J (P_{\theta}) - X I_4) & \n = \n &
    (X + 1) \,\left|\begin{array}{cccc}
       - X & \theta \alpha p_1       & 0 \vspace*{4pt} \\
         0 & (1 - \theta) \alpha p_2 & - X \vspace*{4pt} \\
         0 & (1 - \theta) \alpha q_2 & q_2 / p_2 \end{array} \right| \vspace*{8pt} \\ & \n - \n &
    (X + 1) \,\left|\begin{array}{cccc}
       - X                 & \theta \alpha p_1         & 0 \vspace*{4pt} \\
         q_1 / p_1 + 1 + X & \theta \alpha q_1 - 1 - X & 1 + X \vspace*{4pt} \\
         0                 & (1 - \theta) \alpha p_2   & - X \end{array} \right| \end{array} $$
 Then, expanding along the first column of the first matrix and expanding along the last column of the second matrix, the determinant above becomes
 $$ \begin{array}{rcl}
    \det (\mathcal J (P_{\theta}) - X I_4) & \n = \n &
      - \,X (X + 1) \,\left|\begin{array}{cccc}
         (1 - \theta) \alpha p_2 & - X \vspace*{4pt} \\
         (1 - \theta) \alpha q_2 & q_2 / p_2 \end{array} \right| +
      (X + 1)^2 \,\left|\begin{array}{cccc}
         - X & \theta \alpha p_1 \vspace*{4pt} \\
           0 & (1 - \theta) \alpha p_2 \end{array} \right| \vspace*{8pt} \\ &&
      + \,X (X + 1) \,\left|\begin{array}{cccc}
        - X                 & \theta \alpha p_1 \vspace*{4pt} \\
          q_1 / p_1 + 1 + X & \theta \alpha q_1 - 1 - X \end{array} \right| \end{array} $$
 and a direct calculation gives
 $$ \begin{array}{rcl}
    \det (\mathcal J (P_{\theta}) - X I_4) & \n = \n &
     - X (X + 1)^2 (1 - \theta) \alpha q_2 - X (X + 1)^2 (1 - \theta) \alpha p_2 \vspace*{4pt} \\ &&
     - X (X + 1)(X (\theta \alpha q_1 - 1 - X) + \theta \alpha q_1 + (X + 1) \theta \alpha p_1) \vspace*{4pt} \\ & \n = \n &
     - X (X + 1)^2 (1 - \theta) \alpha + X (X + 1)^2 (X - \theta \alpha) \vspace*{4pt} \\ & \n = \n &
       X (X + 1)^2 (X - \alpha) = X (X + 1)^2 (X - (1 - \lambda_1 p_1)). \end{array} $$
 This shows that~$\sigma (\mathcal J (P_{\theta})) = \{0, -1, \alpha = 1 - \lambda_1 p_1 < 0 \}$.
 In addition,
 $$ \mathcal J (P_{\theta}) \cdot (Q_1 - Q_2) = \mathcal J (P_{\theta}) \cdot \Lambda_1 (p_1, q_1, - p_2, - q_2) = 0, $$
 therefore the eigenspace associated with~0 is generated by the vector connecting~$Q_1$ and~$Q_2$.
\begin{figure}[t!]
\centering
\scalebox{0.42}{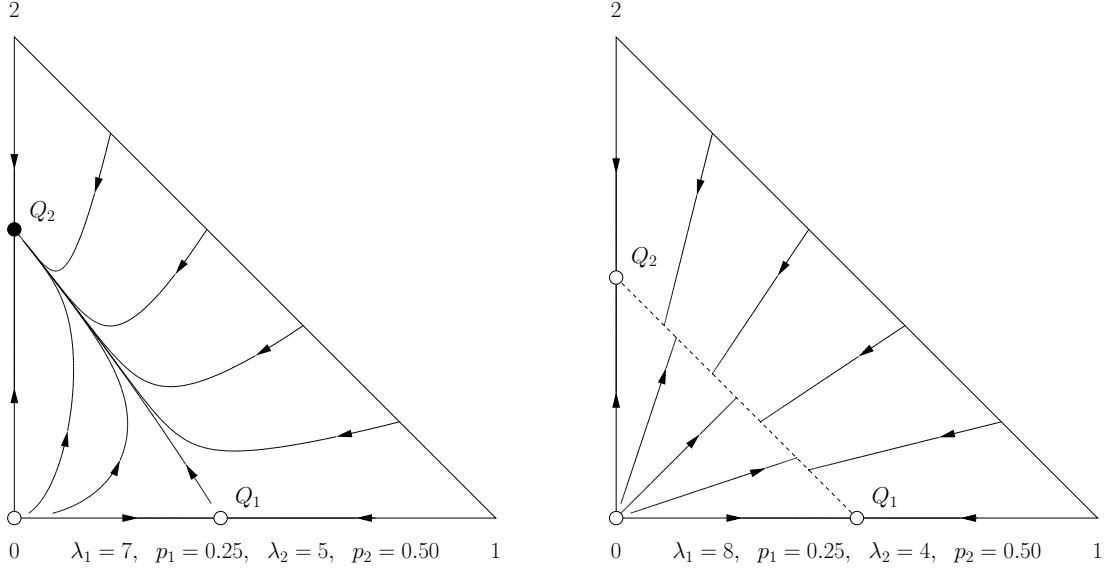}
\caption{\upshape{
 Projection of the solution curves of the mean-field model~\eqref{eq:mf-07}.
 The horizontal axis represents the density of type~1 individuals, while the vertical axis represents the density of type~2 individuals.
 In both pictures, the solution curves start~(and remain) in the invariant subspace~$\Gamma_2$.
 In the left picture,~$\lambda_2 p_2 > \lambda_1 p_1 > 1$ so~$Q_2$ is locally stable, while~$Q_1$ is unstable: type~2 wins.
 In the right picture,~$\lambda_1 p_1 = \lambda_2 p_2 = 2 > 1$ so the set of interior fixed points consists of line segment connecting~$Q_1$ and~$Q_2$: both types coexist.}}
\label{fig:MFt2}
\end{figure}
 Because the other two eigenvalues are negative, we deduce that the segment line of fixed points is locally stable.
 In conclusion, in the presence of two types,
\begin{itemize}
\item
 the population dies out when~$\lambda_1 p_1 < 1$ and~$\lambda_2 p_2 < 1$ in the sense that the trivial fixed point is the only fixed point in the four-dimensional simplex, \vspace*{4pt}
\item
 type~1 individuals win when~$\lambda_1 p_1 > 1$ and~$\lambda_1 p_1 > \lambda_2 p_2$ in the sense that~$Q_1 \in \Delta_4$ and is locally stable, while~$Q_2 \notin \Delta_4$ or is locally unstable, \vspace*{4pt}
\item
 type~2 individuals win when~$\lambda_2 p_2 > 1$ and~$\lambda_2 p_2 > \lambda_1 p_1$ in the sense that~$Q_2 \in \Delta_4$ and is locally stable, while~$Q_1 \notin \Delta_4$ or is locally unstable, \vspace*{4pt}
\item
 coexistence occurs when~$\lambda_1 p_1 = \lambda_2 p_2 > 1$ in the sense that the segment line~$(Q_1, Q_2)$ is a locally stable segment line of interior fixed points.
\end{itemize}
 Note also that, if~$q_1 u_{+1} = p_1 u_{-1}$ initially, it follows from~\eqref{eq:mf-07} that
 $$ \begin{array}{rcl}
      u_{+1}' u_{-1} - u_{-1}' u_{+1} & \n = \n & (\lambda_1 p_1 u_{+1} u_0 - u_{+1}) u_{-1} - (\lambda_1 q_1 u_{+1} u_0 - u_{-1}) u_{+1} \vspace*{4pt} \\
                                      & \n = \n &  \lambda_1 u_{+1} u_0 (p_1 u_{-1} - q_1 u_{+ 1}) = 0 \end{array} $$
 therefore~$q_1 u_{+1} = p_1 u_{-1}$ at all times.
 The same holds by symmetry for type~2.
 This shows in particular that the two-dimensional vector space~$\Gamma_2$ generated by the vectors~$Q_1$ and~$Q_2$ is an invariant subspace of the system~\eqref{eq:mf-07}.
 Figure~\ref{fig:MFt2} shows~(a projection of) the solution curves of the mean-field model in the case where type~2 wins on the left, and in the coexistence case on the right.
 In both pictures, the solution curves start~(and remain) in the invariant subspace~$\Gamma_2$.


\section{Extinction when~$\lambda p \leq \lambda_c$}
\label{sec:coupling}
 This section is devoted to the proof of Theorem~\ref{th:coupling}.
 The first step to study the process and prove our results is to construct the interacting particle system graphically from a collection of independent Poisson processes/exponential clocks, an approach due to Harris~\cite{harris_1978}.
 More precisely, the multitype contact process with sterile states starting from any initial configuration can be constructed using the collection of exponential clocks and updating rules below.
\begin{figure}[t!]
\centering
\scalebox{0.35}{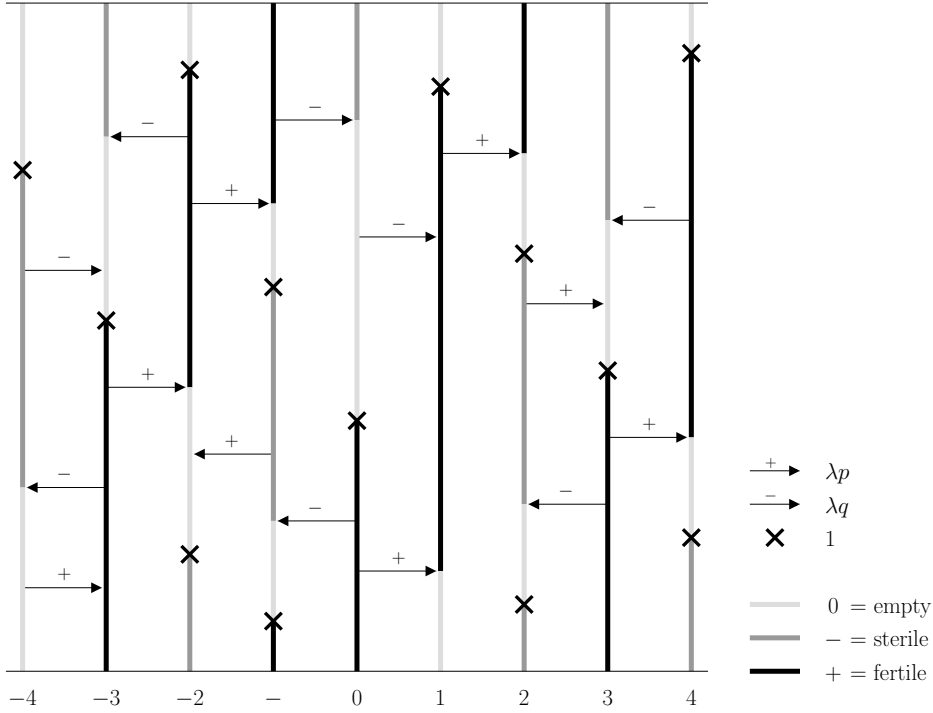}
\caption{\upshape{
 Graphical representation of the single-type process.
 The light gray lines represent the empty sites~(0), the dark gray lines the sterile individuals~($- 1$), and the black lines the fertile individuals~($+ 1$).}}
\label{fig:GR}
\end{figure}
 In the presence of two types, labeled~$i = 1, 2$, for each site~$x \in \Z^d$, and each neighbor~$y$ of site~$x$,
\begin{itemize}
\item
 Place a rate~$\lambda_i p_i / 2d$ exponential clock along~$\vec{xy}$, and draw
 $$ (x, t) \overset{+i}{\xrightarrow{\hspace*{25pt}}} (y, t) \quad \hbox{at the times~$t$ the clock rings} $$
 to indicate that if there is a $+i$ at the tail~$x$ of the arrow and a~0 at the head~$y$ of the arrow just before time~$t$, then site~$y$ becomes type~$+i$ at time~$t$. \vspace*{5pt}
\item
 Place a rate~$\lambda_i q_i / 2d$ exponential clock along~$\vec{xy}$, and draw
 $$ (x, t) \overset{-i}{\xrightarrow{\hspace*{25pt}}} (y, t) \quad \hbox{at the times~$t$ the clock rings} $$
 to indicate that if there is a $+i$ at the tail~$x$ of the arrow and a~0 at the head~$y$ of the arrow just before time~$t$, then site~$y$ becomes type~$-i$ at time~$t$. \vspace*{5pt}
\item
 Place a rate one exponential clock at~$x$, and put
 $$ \times \ \hbox{at} \ (x, t) \quad \hbox{at the times~$t$ the clock rings} $$
 to indicate that if there is an individual of either type at site~$x$ just before time~$t$, then the site becomes empty/type~0 at time~$t$.
\end{itemize}
 Note that the single-type process~$\xi$ with parameters~$\lambda = \lambda_1$ and~$p = p_1$, in which the fertile individuals are labeled~$+ = +1$ and the sterile individuals~$- = -1$, can be constructed by ignoring the type~$\pm 2$ arrows, as shown in Figure~\ref{fig:GR}.
 Let also~$\eta$ be the contact process with parameter~$\lambda p$ with states~0~=~empty and~$+$~=~occupied.
 Theorem~\ref{th:coupling} follows from the next lemma.
\begin{lemma}
\label{lem:coupling}
 There is a coupling~$(\xi, \eta)$ such that~$\,\xi_0 \leq \eta_0 \,\Rightarrow \,\xi_t \leq \eta_t \ \forall \,t > 0$.
\end{lemma}
\begin{proof}
 The process~$\eta$ can be constructed from the same graphical representation as~$\xi$ using the arrows of type~$+$ and the death marks~$\times$, but ignoring the arrows of type~$-$.
 In particular, to prove the lemma, it suffices to show that, using this joint graphical representation,
 $$ (\xi_0, \eta_0) \in S = \{-0, -+, 00, 0+, ++ \}^{\Z^d} \ \Longrightarrow \ (\xi_t, \eta_t) \in S \ \ \forall \,t > 0, $$
 meaning that~$S$ is closed under the dynamics of the coupling.
 Table~\ref{tab} shows all the states that are created by the dynamics~(both types of arrows and the death marks) from the five states above and confirms that~$S$ is closed under the dynamics.
\end{proof} \vspace*{8pt} \\
 Because the contact process~$\eta$ dies out when~$\lambda p \leq \lambda_c$, it follows from the coupling in Lemma~\ref{lem:coupling} that there is also extinction of the fertile individuals in the process~$\xi$.
 Since the sterile individuals can't give birth, this implies that the process~$\xi$ dies out, which proves Theorem~\ref{th:coupling}.
\begin{table}[t!]
\centering
\scalebox{0.38}{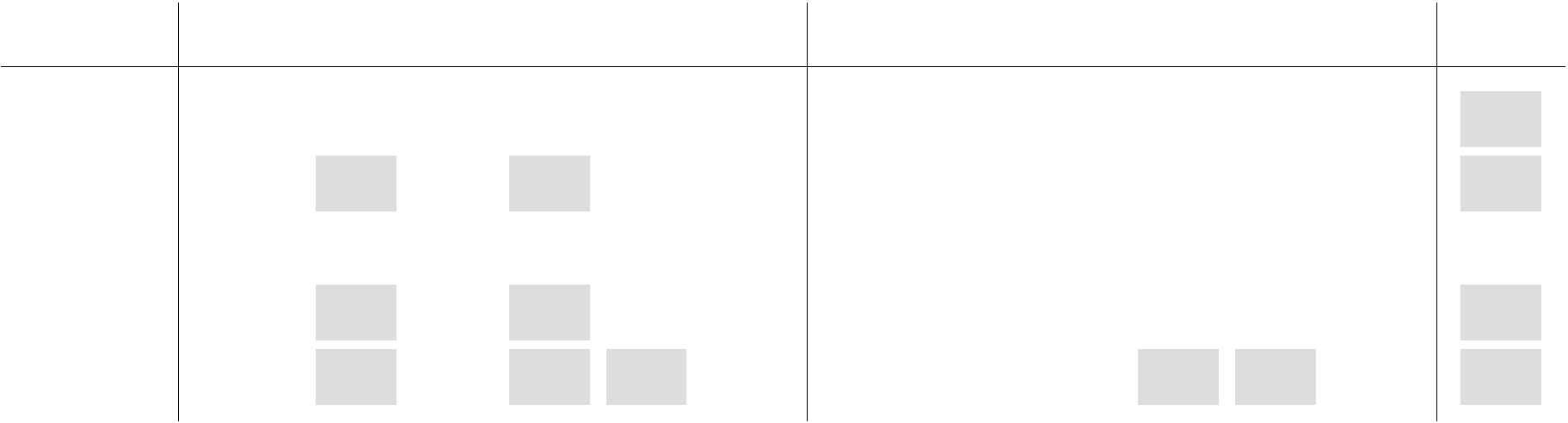}
\caption{\upshape{
 Dynamics of the coupling~$(\xi, \eta)$.
 The states in the first column, respectively, first row, represent the possible states at site~$x$, respectively, at site~$y$ before a potential update, while the states in the table are the states resulting from the interactions.
 The gray boxes underline the interactions resulting in a change of configuration.}}
\label{tab}
\end{table}


\section{Survival when~$\lambda > \lambda_c$ and~$p$ close to~1}
\label{sec:survival}
 This section is devoted to Theorem~\ref{th:survival}, which states that, for all~$\lambda > \lambda_c$, the single-type process survives when~$p$ is close to one.
 The starting point is to observe that, when~$p = 1$, only~$+1$s are produced, therefore our process reduces to the basic contact process~\cite{harris_1974}.
 In the supercritical phase~$\lambda > \lambda_c$, block constructions show that the process properly rescaled in space and time dominates oriented site percolation with a density of open sites arbitrarily close to one~\cite{bezuidenhout_grimmett_1990,durrett_neuhauser_1997}.
 Once the space and time scales are fixed, the contact process with a sterile state matches the basic contact process in any given block with probability arbitrarily close to one when~$p$ is close to one, so the theorem follows from a standard perturbation argument. \\
\indent
 The domination of supercritical percolation by the contact process properly rescaled in space and time was first established by Bezuidenhout and Grimmett~\cite{bezuidenhout_grimmett_1990}, who used their block construction to prove a number of important open problems about the contact process:
 extinction at the critical value, complete convergence theorem, and shape theorem.
 We will use instead the block construction from Durrett and Neuhauser~\cite{durrett_neuhauser_1997}, even though it was designed to study the multitype contact process, because it gives more explicit space and time scales.
 A direct application of their construction gives the following for our process with~$\lambda > \lambda_c$ and~$p = 1$.
 To begin with, let
 $$ \Lat_1 = \{(m, n) \in \Z^d \times \Z_+ : m_1 + \cdots + m_d + n \ \hbox{is even} \}, $$
 which we turn into a directed graph~$\vec \Lat_1$ by putting arrows
 $$ \begin{array}{rcl} (m, n) \to (m', n') & \Longleftrightarrow & |m_1 - m_1'| + \cdots + |m_d - m_d'| = 1 \ \ \hbox{and} \ \ n' = n + 1. \end{array} $$
 The percolation process with parameter~$1 - \ep_1$ is obtained by assuming that each site~$(m, n)$ is open with probability~$1 - \ep_1$ and closed with probability~$\ep_1$, and we say that a site is wet if it can be reached from a directed path of open sites starting at level~$n = 0$.
 Returning to the interacting particle system, let~$L_1$ be a large integer, and consider the space-time blocks
 $$ \begin{array}{rcl}
    \A_{m, n}^1 & \n = \n & (mL_1, nL_1^2) + ([- L, L]^d \times \{0 \}), \vspace*{4pt} \\
    \B_{m, n}^1 & \n = \n & (mL_1, nL_1^2) + ([- 3L_1, 3L_1]^d \times [0, L_1^2]), \end{array} $$
 for all~$(m, n) \in \Lat_1$.
 See Figure~\ref{fig:invade} for a picture.
\begin{figure}[t!]
\centering
\scalebox{0.40}{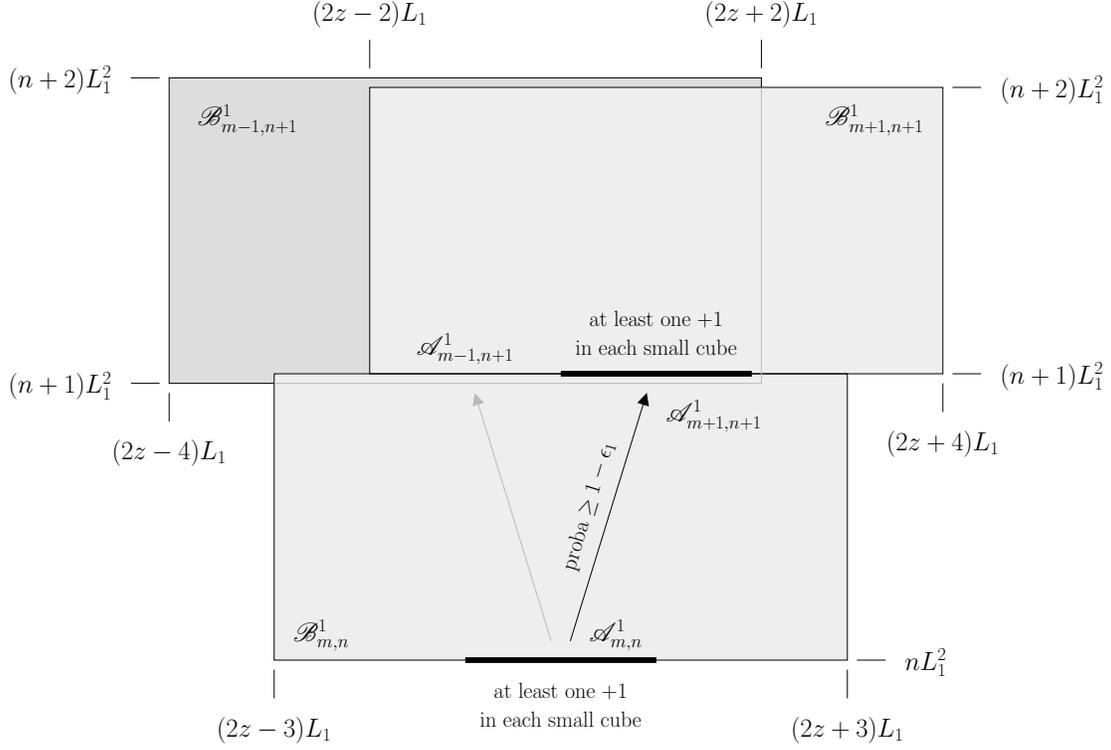}
\caption{\upshape{Picture of the block construction in Durrett and Neuhauser~\cite{durrett_neuhauser_1997}.}}
\label{fig:invade}
\end{figure}
 Note that~$\A_{m, n}^1$ is flat, at the bottom of~$\B_{m, n}^1$.
 Now, partition~$\A_{m, n}^1$ into small cubes of size~$L^{0.1} \times \cdots \times L^{0.1}$, and define the events
 $$ \begin{array}{l}
      E_{m, n}^1 = \{\hbox{each small cube in~$\A_{m, n}^1$ contains at least one~$+1$} \vspace*{4pt} \\ \hspace*{75pt}
                     \hbox{and the bottom of block~$\B_{m, n}^1$ does not contain any~$-1$s} \} \end{array} $$
 for all~$(m, n) \in \Lat_1$.
 Using a repositioning algorithm to show the existence, with probability arbitrarily close to one when~$L_1$ is large, of many dual paths
 $$ \A_{m', n'}^1 \downarrow \A_{m, n}^1 \quad \hbox{for all} \quad (m', n') \leftarrow (m, n), $$
 Durrett and Neuhauser~\cite[Section~3]{durrett_neuhauser_1997} proved that there are good events~$G_{m, n}^1$ that only depend on the graphical representation in the block~$\B_{m, n}^1$ such that, for all~$\ep_1 > 0$ and for all~$\lambda > \lambda_c$, the scale parameter~$L_1 < \infty$ can be chosen sufficiently large that
\begin{equation}
\label{eq:CP-1}
\begin{array}{rl}
\hbox{(a)} & P (G_{m, n}^1) \geq 1 - \ep_1 / 2, \vspace*{4pt} \\
\hbox{(b)} & E_{m, n}^1 \cap G_{m, n}^1 \ \Longrightarrow \ E_{m', n'}^1 \ \ \hbox{for all} \ \ (m', n') \leftarrow (m, n). \end{array}
\end{equation}
 In words, if the flat region~$\A_{m, n}^1$ contains many fertile individuals and the bottom of~$\B_{m, n}^1$ does not contain any sterile individuals, then, with high probability, the same holds for the blocks immediately above.
 In particular, calling~$(m, n)$ a 1-good site if~$E_{m, n}^1$ occurs, and noticing that
 $$ \B_{m, n}^1 \cap \B_{m', n'}^1 = \varnothing \quad \hbox{whenever} \quad |m - m'| > 6 \quad \hbox{or} \quad |n - n'| > 2, $$
 it follows from the comparison result~\cite[Theorem~A.4]{durrett_1995} that
\begin{proposition}
\label{prop:contact}
 For all~$\ep_1 > 0$ and~$\lambda > \lambda_c$, there exists~$L_1 < \infty$ such that the set of 1-good sites dominates the set of wet sites in a 6-dependent oriented site percolation process on the directed graph~$\vec \Lat_1$ with parameter~$1 - \ep_1 / 2$.
\end{proposition}
\noindent
 To prove Theorem~\ref{th:survival}, the next step is to apply a perturbation argument to extend the previous proposition to the process with parameter~$p$ close to one.
\begin{proposition}
\label{prop:survive}
 For all~$\ep_1 > 0$ and~$\lambda > \lambda_c$, there exist~$L_1 < \infty$ and~$p_+ = p_+ (L_1) < 1$ such that, for all~$p > p_+$, the set of 1-good sites dominates the set of wet sites in a 6-dependent oriented site percolation process on the directed graph~$\vec \Lat_1$ with parameter~$1 - \ep_1$.
\end{proposition}
\begin{proof}
 It suffices to find good events~$\bar G_{m, n}^1$ such that~\eqref{eq:CP-1} holds for~$\ep_1$ instead of~$\ep_1 / 2$, and~$p$ close to one instead of~$p = 1$.
 To do this, we let
 $$ \bar G_{m, n}^1 = F_{m, n} \cap G_{m, n}^1 \quad \hbox{where} \quad F_{m, n} = \{\hbox{no type~$-$ arrows in} \ \B_{m, n}^1 \}. $$
 Because, on the event~$F_{m, n}$, the graphical representation of the process matches the graphical representation of the basic contact process in~$\B_{m, n}^1$, it follows from~(\ref{eq:CP-1}.b) that
\begin{equation}
\label{eq:survive1}
  E_{m, n}^1 \cap \bar G_{m, n}^1 \ \Longrightarrow \ E_{m', n'}^1 \ \ \hbox{for all} \ \ (m', n') \leftarrow (m, n).
\end{equation}
 Now, notice that the number of type~$-$ arrows in~$\B_{m, n}^1$ is equal to
 $$ Z_1 = \poisson (\Lambda_1 (1 - p)) \quad \hbox{where} \quad \Lambda_1 = \lambda (6L_1 + 1)^d L_1^2. $$
 In particular, the probability of~$F_{m, n}$ satisfies
\begin{equation}
\label{eq:survive2}
  P (F_{m, n}) = P (Z_1 = 0) = e^{- \Lambda_1 (1 - p)} \geq 1 - \ep_1 / 2
\end{equation}
 whenever the parameter~$p$ is larger than
 $$ p_+ = 1 + \ln (1 - \ep_1 / 2) / \Lambda_1 < 1. $$
 Combining~(\ref{eq:CP-1}.a) and~\eqref{eq:survive2}, we deduce that, for all~$p > p_+$,
\begin{equation}
\label{eq:survive3}
  P (\bar G_{m, n}^1) \geq 1 - P (F_{m, n}^c) - P ((G_{m, n}^1)^c) \geq 1 - \ep_1 / 2 - \ep_1 / 2 = 1 - \ep_1.
\end{equation}
 The proposition follows from~\eqref{eq:survive1} and~\eqref{eq:survive3}.
\end{proof} \vspace*{8pt} \\
 The last ingredient to deduce the theorem is the following percolation result whose proof relies on a contour argument and can be found in~\cite[Theorem~A.1]{durrett_1995}.
\begin{lemma}
\label{lem:contour}
 Let~$\ep_1 = 6^{-4 \times 13^2}$.
 Then, the 6-dependent oriented site percolation process on the directed graph~$\vec \Lat_1$ with parameter~$1 - \ep_1$ is supercritical, i.e., the cluster of wet sites starting at the origin is infinite with positive probability.
\end{lemma}
\noindent
 To deduce Theorem~\ref{th:survival}, fix~$\ep_1$ like in Lemma~\ref{lem:contour}, then~$L_1 < \infty$ like in Proposition~\ref{prop:survive}.
 By the proposition, the set of~1-good sites, corresponding to blocks occupied by fertile individuals, dominates the set of wet sites in the percolation process with parameter~$1 - \ep_1$.
 Note also that, starting from a translation-invariant product measure with a positive density of~$+1$s, there must be infinitely many~1-good sites at level~$n = 0$.
 Finally, because the percolation process is supercritical according to the lemma, there is survival of the~$\pm 1$s in the interacting particle system.

\section{Extinction when~$p < 1/4d$}
\label{sec:extinction}
 In this section, we prove an exponential decay of the process when~$p < 1/4d$ regardless of the value of the birth rate.
 We first show that, starting with a single~1, the total number of~1s in the process is dominated by the total progeny of a subcritical Galton-Watson branching process that itself exhibits an exponential decay, from which we can deduce an exponential decay in space and time of the particle system.
 To extend the result to the process starting from any initial configuration, we use a block construction to couple the process with oriented site percolation in which sites are open with probability arbitrarily close to one.
 The presence of a~$\pm 1$ at a time that scales like~$n$ in the particle system implies the presence of a path of closed sites of length at least~$n$ in the percolation process whose probability decreases exponentially with~$n$. \\
\indent
 To begin with, we study the process~$\xi^{(0, 0)}$ starting with a single~1 at the origin while all the other sites are empty.
 Let~$|\xi^{(0, 0)}|$ denote the total number of~1s~(including the initial~1) in the process across space and time until possible extinction.
\begin{lemma}
\label{lem:offspring}
 The random variable~$|\xi^{(0, 0)}|$ is dominated by the total progeny of the Galton-Watson branching process whose offspring distribution~$Y$ has probability-generating function
 $$ G_Y (s) = \frac{\bar p (ps + q)^{2d}}{1 - (1 - \bar p)(ps + q)} \quad \hbox{where} \quad \bar p = \frac{1}{2d + 1}. $$
\end{lemma}
\begin{proof}
 Because individuals of either type die independently at rate one, it follows from the superposition property for Poisson processes that the number of deaths~(number of times an occupied site becomes empty) in the neighborhood of a~1 during its lifespan is dominated by the shifted geometric random variable~$N$ with success probability
 $$ \bar p = 1 / (\hbox{number of neighbors} + 1) = 1 / (2d + 1). $$
 This implies that the total number of offspring~(number of births of~$\pm 1$s onto an empty site) produced by a single~1 is dominated by the random variable
 $$ 2d + N = 2d + \geometric (\bar p) - 1. $$
 Since in addition each offspring is independently fertile~(of type~1) with probability~$p$, we deduce that the number of~1s produced by a single~1 is dominated by
 $$ Y = Y_1 + Y_2 + \cdots + Y_{2d + N} \quad \hbox{where} \quad Y_i = \hbox{independent} \bernoulli (p). $$
 The probability-generating functions of~$Y_i$ and~$N$ are given by
 $$ G_{Y_i} (s) = E (s^{Y_i}) = ps + q \quad \hbox{and} \quad G_N (s) = E (s^N) = \bar p / (1 - (1 - \bar p) s). $$
 Using also the independence of the random variables, we conclude that
 $$ \begin{array}{rcl}
      G_Y (s) & \n = \n & E (s^Y) = E (E(s^Y \,| \,N)) = E ((ps + q)^{2d + N}) \vspace*{-2pt} \\
              & \n = \n & \displaystyle (ps + q)^{2d} E ((ps + q)^N) = (ps + q)^{2d} G_N (ps + q) = \frac{\bar p (ps + q)^{2d}}{1 - (1 - \bar p)(ps + q)}, \end{array} $$
 which proves the lemma.
\end{proof} \vspace*{8pt} \\
 Note that, when~$p < 1/4d$, the expected number of offspring is
 $$ E (Y) = p (2d + E (N)) = p (2d + 1 / \bar p - 1) = p( 2d + 2d + 1 - 1) = 4dp < 1, $$
 therefore the process starting from a single~1 dies out by the lemma.
 To deal with more general initial configurations, for all~$x \in \Z^d$, let~$\xi^{(x, 0)}$ be the process starting at time~0 with a single~1 at site~$x$ constructed from the same graphical representation as~$\xi$.
 Then, each of these processes dies out, while the process~$\xi$ is also dominated by the superposition of these non-interacting processes~(that allows multiple~1s per site), so the process~$\xi$ dies out as well.
 We now improve this result, showing an exponential decay that we will need later to study the multitype process.
\begin{lemma}
\label{lem:progeny}
 Let~$p < 1/4d$ and let~$X$ be the total progeny in the Galton-Watson branching process introduced above.
 Then, there exist~$C_1 < \infty$ and~$s_1 = s_1 (p, d) > 1$ such that
 $$ P (X > n) \leq C_1 s_1^{-n} \quad \hbox{for all} \quad n \in \N. $$
\end{lemma}
\begin{proof}
 According to~\cite[Section~1.13]{harris_1963}, the probability-generating function~$G_X$ of the total progeny is obtained from~$G_Y$ through the relationship
 $$ G_X (s) = s \,G_Y (G_X (s)). $$
 Recalling~$G_Y$ from Lemma~\ref{lem:offspring}, this becomes
 $$ (1 - (1 - \bar p)(p \,G_X (s) + q)) \,G_X (s) = s \bar p (p \,G_X (s) + q)^{2d}. $$
 Taking the derivative on both sides, we get
 $$ \begin{array}{l}
      - (1 - \bar p) p \,G_X' (s) \,G_X (s) + (1 - (1 - \bar p)(p \,G_X (s) + q)) \,G_X' (s) \vspace*{4pt} \\ \hspace*{80pt}
      = \bar p (p \,G_X (s) + q)^{2d} + 2ds \bar p p G_X' (s) \,(p \,G_X (s) + q)^{2d - 1}. \end{array} $$
 Taking~$s = 1$, and using that~$G_X (1) = 1$ and~$p + q = 1$, we get
 $$ - (1 - \bar p) p \,G_X' (1) + \bar p \,G_X' (1) = \bar p + 2d \bar p p \,G_X' (1). $$
 Recalling that~$\bar p = 1 / (2d + 1)$ and~$p < 1/4d$, we deduce that
 $$ G_X' (1) = \frac{\bar p}{\bar p - (2d \bar p + (1 - \bar p)) p} = \frac{\bar p}{\bar p - (1 + (2d - 1) \bar p) p} = \frac{1}{1 - 4dp} > 0, $$
 showing that there exists~$s_1 = s_1 (p, d) > 1$ such that~$C_1 = G_X (s_1)$ is finite and positive.
 Finally, because~$s_1 > 1$, it follows from Markov's inequality that
 $$ P (X > n) = P (s_1^X > s_1^n) \leq E (s_1^X) / s_1^n = G_X (s_1) \,s_1^{-n} = C_1 s_1^{-n}, $$
 which proves an exponential decay of the progeny.
\end{proof} \vspace*{8pt} \\
 Using the previous two lemmas, we can now prove an exponential decay in space and time for the contact process~$\xi^{(0, 0)}$ when~$p < 1/4d$ regardless of the value of the birth rate~$\lambda$.
\begin{lemma}
\label{lem:decay}
 Let~$p < 1/4d$.
 Then, there exist~$C_2 < \infty$~and $s_2 > 1$ such that
 $$ P (\xi^{(0, 0)} \not \subset [-n, n]^d \times [0, n]) \leq C_2 s_2^{-n} \quad \hbox{for all $n \in \N$ and all~$\lambda \in [0, \infty]$}. $$
\end{lemma}
\begin{proof}
 For simplify, we only look at the probability that the fertile population escapes from the space-time block.
 Because sterile individuals can't give birth and each fertile individual gives birth to at most~$Y$ sterile individuals, similar estimates hold for the total population.
 To prove exponential decay of the radius of the process, note that, due to nearest neighbor interactions, for the~1s to travel a distance~$n$, there must be at least~$n$ fertile individuals.
 This, together with Lemmas~\ref{lem:offspring}--\ref{lem:progeny}, implies that there exist~$C_1 < \infty$ and~$s_1 > 1$ such that
\begin{equation}
\label{eq:decay1}
  P (\xi^{(0, 0)} \not \subset [-n, n]^d \times [0, \infty)) \leq P (|\xi^{(0, 0)}| > n) \leq P (X > n) \leq C_1 s_1^{-n}.
\end{equation}
 To prove exponential decay of the time to extinction, call individual~$i$ the~$i$th individual that appears in the process~(if this individual exists) and let~$T_i$ be the lifespan of this individual.
 Because individual~$i$ appears in the system before all the previous individuals die, the time to extinction is less than the sum of the lifespans, which implies that
\begin{equation}
\label{eq:decay2}
  P (\xi^{(0, 0)} \not \subset \Z^d \times [0, n] \,| \,|\xi^{(0, 0)}| = k \leq n/2) = P (T = T_1 + T_2 + \cdots + T_k > n).
\end{equation}
 Using Markov's inequality and that~$T_i = \exponential (1)$,
\begin{equation}
\label{eq:decay3}
\begin{array}{rcl}
  P (T > n) & \n = \n & P (e^{T/2} > e^{n/2}) \leq E (e^{T/2}) \,e^{- n/2} = G_T (\sqrt{e}) \,e^{- n/2} \vspace*{4pt} \\
            & \n = \n & (G_{T_i} (\sqrt{e}))^k \,e^{- n/2} \leq (G_{T_i} (\sqrt{e}))^{n/2} \,e^{- n/2} \vspace*{4pt} \\
            & \n = \n & (1 - \ln (\sqrt{e}))^{- n/2} e^{- n/2} = (e/2)^{- n/2}. \end{array}
\end{equation}
 Combining~\eqref{eq:decay1}--\eqref{eq:decay3} and using again Lemmas~\ref{lem:offspring}--\ref{lem:progeny}, we conclude that
 $$ \begin{array}{rcl}
      P (\xi^{(0, 0)} \not \subset [-n, n]^d \times [0, n]) & \n \leq \n &
      P (\xi^{(0, 0)} \not \subset [-n, n]^d \times [0, \infty) \ \hbox{or} \ \xi^{(0, 0)} \not \subset \Z^d \times [0, n]) \vspace*{4pt} \\ & \n \leq \n &
      P (\xi^{(0, 0)} \not \subset [-n, n]^d \times [0, \infty)) + P (\xi^{(0, 0)} \not \subset \Z^d \times [0, n]) \vspace*{4pt} \\ & \n \leq \n &
      P (\xi^{(0, 0)} \not \subset [-n, n]^d \times [0, \infty)) \vspace*{4pt} \\ && \hspace*{0pt} + \
      P (|\xi^{(0, 0)}| > n/2) + P (\xi^{(0, 0)} \not \subset \Z^d \times [0, n] \,| \,|\xi^{(0, 0)}| \leq n/2) \vspace*{4pt} \\ & \n \leq \n &
      C_1 s_1^{-n} + C_1 s_1^{-n/2} + (e/2)^{- n/2}. \end{array} $$
 Because~$s_1 > 1$ and~$e/2 > 1$, this proves the lemma.
\end{proof} \vspace*{8pt} \\
 In preparation for the block construction, we now let~$L_2$ be a large integer to be fixed later, and consider the two space-time blocks
\begin{figure}[t!]
\centering
\scalebox{0.40}{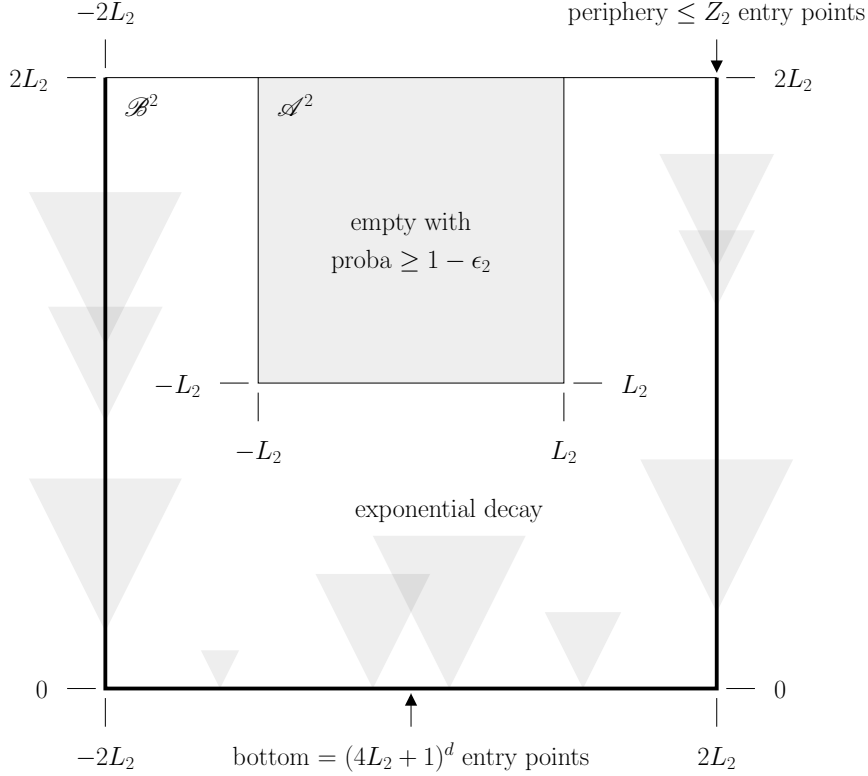}
\caption{\upshape{Picture of the space-time blocks~$\A^2$ and~$\B^2$.}} 
\label{fig:empty}
\end{figure}
 $$ \A^2 = [-L_2, L_2]^d \times [L_2, 2L_2] \quad \hbox{and} \quad \B^2 = [-2L_2, 2L_2]^d \times [0, 2L_2]. $$
 In the next lemma, the bottom and the periphery of block~$\B^2$ refer to the sets
 $$ \begin{array}{rcl}
       \hbox{bottom of~$\B^2$} & \n = \n & [-2L_2, 2L_2]^d \times \{0 \}, \vspace*{4pt} \\
    \hbox{periphery of~$\B^2$} & \n = \n & \{x \in \Z^d : \norm{x}_{\infty} = 2L_2 \} \times [0, 2L_2]. \end{array} $$
 See Figure~\ref{fig:empty} for a picture of these sets.
\begin{lemma}
\label{lem:empty}
 For all~$\ep_2 > 0$, there exists~$L_2 < \infty$ such that~$P (\xi \cap \A^2 \neq \varnothing) \leq \ep_2$.
\end{lemma}
\begin{proof}
 We prove that the result holds regardless of the configuration of the process outside the space-time block~$\B^2$, including the initial configuration, meaning in particular that our estimates only depend on the graphical representation of the process inside block~$\B^2$.
 Having a realization of the graphical representation, we define the following processes.
\begin{itemize}
\item
 For each~$(x, 0)$ at the bottom of~$\B^2$, let~$\xi^{(x, 0)}$ be the process starting at time~0 with a single~1 at site~$x$ constructed from the same graphical representation as~$\xi$. \vspace*{4pt}
\item
 For each arrow from outside~$\B^2$ to a space-time point~$(x, t)$ on the periphery of~$\B^2$ that results in an actual birth, let~$\xi^{(x, t)}$ be the process starting at time~$t$ with a single~1 at site~$x$ constructed from the same graphical representation as~$\xi$.
\end{itemize}
 The superposition of these non-interacting processes~(that allows multiple~1s per site) dominates the process~$\xi$.
 In particular, it suffices to show that, with probability close to one, none of these processes invades block~$\A^2$.
 The first step is to control the number of such processes, corresponding to the number of entry points.
 There are~$(4L_2 + 1)^d$ sites at the bottom of~$\B^2$.
 The periphery has~$2d$ faces, each having~$(4L_2 + 1)^{d - 1}$ sites and height~$2L_2$ units of time.
 In addition, by the same argument as in the proof of Lemma~\ref{lem:offspring}, there is at most one successful birth between two consecutive death marks, so the number of entry points on the periphery is dominated by
 $$ Z_2 = \poisson (\Lambda_2) \ \hbox{with} \ \Lambda_2 = 2d \times (4L_2 + 1)^{d - 1} \times 2L_2 \times 1 = 4dL_2 (4L_2 + 1)^{d - 1}. $$
 Using again Markov's inequality, we get
 $$ P (Z_2 > 2 \Lambda_2) = P (s^{Z_2} > s^{2 \Lambda_2}) \leq G_{Z_2} (s) s^{- 2 \Lambda_2} = e^{- \Lambda_2 (1 - s)} s^{- 2 \Lambda_2} $$
 for all~$s > 1$.
 Taking~$s = 2$, this becomes
\begin{equation}
\label{eq:empty1}
  P (Z_2 > 2 \Lambda_2) \leq e^{- \Lambda_2 (1 - 2)} \,2^{- 2 \Lambda_2} = (4/e)^{- \Lambda_2}.
\end{equation}
 Now, notice that all the entry points at the bottom or around the periphery of~$\B^2$ are at space-time distance at least~$L_2$ of block~$\A^2$, therefore
\begin{equation}
\label{eq:empty2}
\xi \cap \A^2 \neq \varnothing \ \Longrightarrow \ \xi^{(x, t)} \not \subset (x, t) + ([-L_2, L_2]^d \times [0, L_2])
\end{equation}
 for at least one entry point~$(x, t)$.
 Combining~\eqref{eq:empty1}--\eqref{eq:empty2} and Lemma~\ref{lem:decay}, we get
 $$ \begin{array}{rcl}
      P (\xi \cap \A^2 \neq \varnothing) & \n \leq \n &
      P (Z_2 > 2 \Lambda_2) + P (\xi \cap \A^2 \neq \varnothing \,| \,Z_2 \leq 2 \Lambda_2) \vspace*{4pt} \\ & \n \leq \n &
      (4/e)^{- \Lambda_2} + ((4L_2 + 1)^d + 2 \Lambda_2) P (\xi^{(0, 0)} \not \subset [-L_2, L_2]^d \times [0, L_2]) \vspace*{4pt} \\ & \n \leq \n &
      (4/e)^{- \Lambda_2} + ((4L_2 + 1)^d + 2 \Lambda_2) \,C_2 s_2^{-L_2},
    \end{array} $$
 which can be made less than~$\ep_2$ for all~$L_2$ large.
\end{proof} \vspace*{8pt} \\
 To deduce an exponential decay of the population from the previous lemma, we now couple the process properly rescaled in space and time with oriented site percolation.
 Let~$\Lat_2 = \Z^d \times \N$, which we turn into a directed graph~$\vec \Lat_2$ by putting arrows
 $$ \begin{array}{rcl} (m, n) \to (m', n') & \Longleftrightarrow & |m_1 - m_1'| + \cdots + |m_d - m_d'| + |n - n'|= 1 \ \ \hbox{and} \ \ n \leq n'. \end{array} $$
 In other words, starting from each site~$(m, n)$, there are~$2d$ horizontal arrows and one vertical arrow going upward.
 The percolation process with parameter~$1 - \ep_2$ is obtained by assuming that each site~$(m, n)$ is open with probability~$1 - \ep_2$ and closed with probability~$\ep_2$, and we say that a site is wet if it can be reached from a directed path of open sites starting at level~$n = 0$.
 Returning to the interacting particle system, we introduce the collection of space-time blocks
 $$ \A_{m, n}^2 = (2mL_2, nL_2) + \A^2 \quad \hbox{and} \quad \B_{m, n}^2 = (2mL_2, nL_2) + \B^2 $$
 for all~$(m, n) \in \Lat_2$ and the collection of events
 $$ E_{m, n}^2 = \{\xi_t (x) = 0 \ \hbox{for all} \ (x, t) \in \A_{m, n}^2 \} = \hbox{block~$\A_{m, n}^2$ is empty} $$
 for all~$(m, n) \in \Lat_2$.
 Because the distribution of the graphical representation is translation-invariant in space and time, Lemma~\ref{lem:empty} applies to all these events.
 More precisely, there are good events~$G_{m, n}^2$ that only depend on the graphical representation in~$\B_{m, n}^2$ such that, for all~$\ep_2 > 0$, the scale parameter~$L_2 < \infty$ can be chosen in such a way that
 $$ P (G_{m, n}^2) \geq 1 - \ep_2 \quad \hbox{and} \quad G_{m, n}^2 \subset E_{m, n}^2 \quad \hbox{for all} \quad p < 1/4d. $$
 In particular, calling~$(m, n)$ a 2-good site if~$E_{m, n}^2$ occurs, and observing that
 $$ \B_{m, n}^2 \cap \B_{m', n'}^2 = \varnothing \quad \hbox{whenever} \quad |m - m'| > 2 \quad \hbox{or} \quad |n - n'| > 2, $$
 we obtain the following proposition.
\begin{proposition}
\label{prop:decay}
 For all~$\ep_2 > 0$ and~$p < 1/4d$, there exists~$L_2 < \infty$ such that the set of 2-good sites dominates the set of open sites in a 2-dependent oriented site percolation process on the directed graph~$\vec \Lat_2$ with parameter~$1 - \ep_2$.
\end{proposition}
\noindent
 To deduce an exponential decay from the previous proposition, the last step is to prove the lack of percolation of the closed sites when~$\ep_2$ is small enough.
\begin{lemma}
\label{lem:perco}
 Let~$\ep_2 = (4d + 2)^{-5^{d + 1}}$.
 Then, the probability of a directed path of closed sites of length~$n$ starting from the origin is bounded by~$(1/2)^n$.
\end{lemma}
\begin{proof}
 Because there are~$2d + 1$ arrows starting from each site,
\begin{equation}
\label{eq:perco1}
\hbox{\#\,self-avoiding paths of length~$n$ starting from the origin} \leq (2d + 1)^n.
\end{equation}
 In addition, from each self-avoiding path~$\vec \pi = (\pi_1, \pi_2, \ldots, \pi_n) \subset \Lat_2$ of length~$n$, one can extract a subset of~$n / 5^{d + 1}$ sites that are distance~$> 2$ apart, therefore, by~2-dependence,
\begin{equation}
\label{eq:perco2}
  P (\vec \pi \ \hbox{is closed}) \leq \ep_2^{n / 5^{d + 1}} \leq (4d + 2)^{-n}.
\end{equation}
 Combining~\eqref{eq:perco1}--\eqref{eq:perco2}, we deduce that the probability of a directed path of closed sites of length~$n$ starting from the origin is bounded by
 $$ (2d + 1)^n \,P (\vec \pi \ \hbox{is closed}) \leq (2d + 1)^n / (4d + 2)^n = (1/2)^n. $$
 This proves the lemma.
\end{proof} \vspace*{8pt} \\
 To deduce Theorem~\ref{th:extinction}, fix~$\ep_2$ like in Lemma~\ref{lem:perco}, then~$L_2 < \infty$ like in Proposition~\ref{prop:decay}.
 By the proposition, because individuals can't appear spontaneously, the presence of a~$\pm 1$ in space-time block~$\A_{m, n}^2$ implies the existence of a path of closed sites leading to~$(m, n)$ in the percolation process.
 Now, by the lemma, the probability that such a path exists decays exponentially with~$n$, from which it follows that, for all~$p < 1/4d$, there is an exponential decay of the~$\pm 1$s.


\section{The multitype process}
\label{sec:mcp}
 This section is devoted to the proof of Theorem~\ref{th:mcp} about the multitype process.
 The basic idea is to use the exponential decay of the~$\pm 2$s to prove the existence of a linearly growing cluster of~$\pm 1$s that does not interact with the surrounding~$\pm 2$s.
 In particular, the proof relies on Theorems~\ref{th:survival}--\ref{th:extinction} and their proofs.
 From now on, the process~$\xi$ refers to the multitype process constructed from the graphical representation described in Section~\ref{sec:coupling}, starting from a translation-invariant product measure with a positive density of both types.
 The parameters are given by
 $$ \lambda_1 > \lambda_c \quad \hbox{and} \quad p_1 = 1 \quad \hbox{and} \quad p_2 < 1/4d, $$
 while the birth rate~$\lambda_2$ is arbitrary.
 Because the~$\pm 2$s die out in the absence of~$\pm 1$s according to Theorem~\ref{th:extinction}, and they are now blocked by the surrounding~$\pm 1$s, they again die out.
 However, even if the~$\pm 1$s survive in the absence of~$\pm 2$s according to Theorem~\ref{th:survival}, because they are now blocked by the surrounding~$\pm 2$s, their survival is unclear.
 To prove that the~$\pm 1$s win, we let
\begin{itemize}
\item $\ep_1 = 6^{-4 \times 13^2} > 0$ then~$L_1$ such that Proposition~\ref{prop:survive} holds for this~$\ep_1$, \vspace*{4pt}
\item $\ep_2 = (4d + 2)^{-5^{d + 1}} > 0$ then~$L_2$ such that Proposition~\ref{prop:decay} holds for this~$\ep_2$.
\end{itemize}
 We may also assume that~$L_1 = L_2$ since the propositions hold for all~$L_1, L_2$ large.
 From ow on, we denote by~$L$ their common value.
 We also let~$K$ be an integer to be fixed later, and consider two collections of single-type processes
 $$ \{\eta^{z, 1} : z \in \Z^d \} \quad \hbox{and} \quad \{\eta^{z, 2} : z \in \Z^d \} $$
 constructed from the same graphical representation as~$\xi$ but starting from
 $$ \begin{array}{rcl}
    \eta_0^{z, 1} (x) & \n = \n & \left\{\begin{array}{cl} \xi_0 (x) & \hbox{if} \ \xi_0 (x) = \pm 1 \ \hbox{and} \ x \in 2zL + [-L, L]^d, \vspace*{4pt} \\
                                                               0     & \hbox{otherwise}, \end{array} \right. \vspace*{8pt} \\
    \eta_0^{z, 2} (x) & \n = \n & \left\{\begin{array}{cl} \xi_0 (x) & \hbox{if} \ \xi_0 (x) = \pm 2 \ \hbox{and} \ x \notin 2zL + [-KL, KL]^d, \vspace*{4pt} \\
                                                               0     & \hbox{otherwise}. \end{array} \right. \end{array} $$
 We also define another collection of processes~$\bar \eta^{z, 1}$ as follows: letting
 $$ \begin{array}{rcl}
    \nabla_{z, 1} & \n = \n & \{(m, n) \in \Lat_1 : \hbox{there is a path~$(2z, 0) \to (m, n)$ in the graph~$\vec \Lat_1$} \} \vspace*{4pt} \\
                  & \n = \n & \{(m, n) \in \Lat_1 : m \in 2z + [-n, n]^d \}, \end{array} $$
 the process~$\bar \eta^{z, 1}$ is the process~$\eta^{z, 1}$ modified so that the~1s outside
 $$ \nabla_z = \bigcup_{(m, n) \in \nabla_{z, 1}} (mL, nL^2) + ([-3L, 3L]^d \times [0, L^2]) $$
 are instantaneously killed.
 Figure~\ref{fig:cone} shows a picture of the sets~$\nabla_{z, 1}$ and~$\nabla_z$.
\begin{figure}[t!]
\centering
\scalebox{0.40}{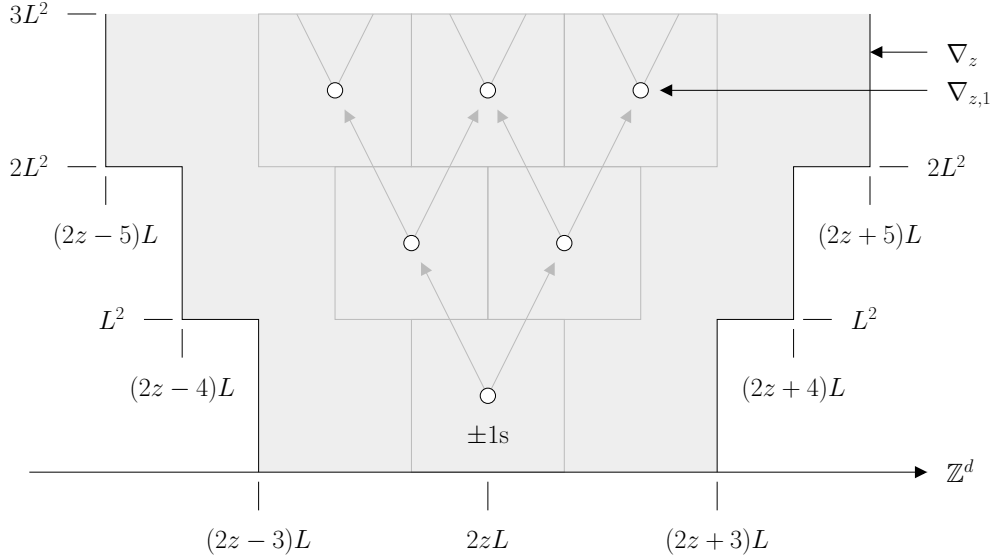}
\caption{\upshape{Picture of the sets~$\nabla_{z, 1}$ and~$\nabla_z$.}}
\label{fig:cone}
\end{figure}
 To study the processes~$\eta^{z, 2}$, we also define the sets
 $$ \nabla_{z, 2} = \{(m, n) \in \Lat_2 : (mL, nL) \in \nabla_z \}. $$
 Having a realization of the graphical representation of~$\xi$ and a realization of its initial configuration, we say that a site~$z \in \Z^d$ is a~1-source if the following three events occur:
 $$ \begin{array}{rcl}
      A_z & \n = \n & \{\bar \eta_t^{z, 1} \neq \varnothing \ \forall \,t \}, \vspace*{4pt} \\
      B_z & \n = \n & \{\eta^{z, 2} \cap \nabla_z = \varnothing \}, \vspace*{4pt} \\
      C_z & \n = \n & \{\xi_0 (x) \neq \pm 2 \ \forall \,x \in 2zL + [-KL, KL]^d \}. \end{array} $$
 In words, the process restricted to~$\nabla_z$ that keeps track of the~1s survives, the process that keeps track of the~2s does not come in contact with~$\nabla_z$, and the multitype contact process starts from a configuration that has no~$\pm 2$s in the spatial region~$2zL + [-KL, KL]^d$.
 By translation invariance, the probabilities~$P (C_z)$ do not depend on~$z$ and are positive.
 The next two lemmas show that the same property is satisfied for the events~$A_z$ and~$B_z$.
\begin{lemma}
\label{lem:1-expand}
  $P (A_z) > 0$.
\end{lemma}
\begin{proof}
 Starting with a positive density of~1s, site~$(2z, 0)$ is 1-good with positive probability.
 In addition, it follows from the proof of~\cite[Theorem~2]{durrett_neuhauser_1997} and Theorem~\ref{th:survival} that, even if the~1s outside the space-time region~$\nabla_z$ are instantaneously killed, the set of 1-good sites dominates the set of wet sites in the percolation process of Proposition~\ref{prop:survive}.
 In particular, letting~$\C_{2z}$ be the set of~$(m, n) \in \Lat_1$ that can be reached from~$(2z, 0)$ by an open path, our choice of~$\ep_1$ implies that
 $$ \begin{array}{rcl}
      P (A_z) & \n \geq \n & P (A_z \,| \,(2z, 0) \ \hbox{is 1-good}) P ((2z, 0) \ \hbox{is 1-good}) \vspace*{4pt} \\
              & \n   =  \n & P (\bar \eta_t^{z, 1} \neq \varnothing \ \forall \,t \,| \,(2z, 0) \ \hbox{is 1-good}) P ((2z, 0) \ \hbox{is 1-good}) \vspace*{4pt} \\
              & \n   =  \n & P (|\C_{2z}| = \infty) P ((2z, 0) \ \hbox{is 1-good}) > 0. \end{array} $$
 This proves the lemma.
\end{proof}
\begin{lemma}
\label{lem:2-decay}
 There exists~$K < \infty$ such that~$P (B_z) > 0$.
\end{lemma}
\begin{proof}
 If~$\eta^{z, 2}$ comes in contact with~$\nabla_z$, there must be a path of closed sites
\begin{equation}
\label{eq:2-decay-1}
  (x, 0) \to \nabla_{z, 2} \quad \hbox{for some} \quad x \notin 2z + [-K, K]^d
\end{equation}
 in the percolation process introduced in Proposition~\ref{prop:decay}.
 Because~$L_1 = L_2 = L$, this path must have length at least~$\norm{x - 2z}_{\infty} - 3$, so Lemma~\ref{lem:perco} implies that
\begin{equation}
\label{eq:2-decay-2}
  P (\hbox{closed path~$(x, 0) \to \nabla_{z, 2}$}) \leq (1/2)^{\norm{x - 2z}_{\infty} - 3}.
\end{equation}
 Also, the number of vertices on the sphere~$S (2z, r)$ is bounded by
\begin{equation}
\label{eq:2-decay-3}
\card (S (2z, r)) = \card \{x \in \Z^d : \norm{x - 2z}_{\infty} = r \} \leq 2d (2r + 1)^{d - 1}.
\end{equation}
 Combining~\eqref{eq:2-decay-1}--\eqref{eq:2-decay-3}, we deduce that
 $$ \begin{array}{rcl}
      P (B_z^c) & \n = \n & \displaystyle P (\eta^{z, 2} \cap \nabla_z \neq \varnothing) \leq
                            \displaystyle \sum_{x \notin 2z + [- K, K]^d} P (\hbox{closed path~$(x, 0) \to \nabla_{z, 2}$}) \vspace*{4pt} \\
                & \n = \n & \displaystyle \sum_{r > K} \ \sum_{x \in S (2z, r)} P (\hbox{closed path~$(x, 0) \to \nabla_{z, 2}$}) \vspace*{8pt} \\
                & \n \leq \n & \displaystyle \sum_{r > K} \ 2d (2r + 1)^{d - 1} (1/2)^{r - 3}. \end{array} $$
 Because the last sum is less than one for~$K < \infty$ large, the lemma follows.
\end{proof} \vspace*{8pt} \\
 To deduce the theorem, the last step is to prove the existence of a~1-source.
 In fact, combining the previous two lemmas, we obtain the existence of infinitely many.
\begin{lemma}
\label{lem:1-source}
 With probability one, $\card \{z \in \Z^d : z \ \hbox{is a 1-source} \} = \infty$.
\end{lemma}
\begin{proof}
 Note that the event~$A_z$, respectively, the event~$B_z$, only depends on the graphical representation inside~$\nabla_z$, respectively, outside~$\nabla_z$.
 In addition, the event~$C_z$ only depends on the initial configuration, which is independent of the graphical representation.
 In particular, all three events are independent, so it follows from Lemmas~\ref{lem:1-expand}--\ref{lem:2-decay} that
 $$ P (z \ \hbox{is a 1-source}) = P (A_z \cap B_z \cap C_z) = P (A_z) P (B_z) P (C_z) > 0. $$
 Since in addition the initial configuration and the graphical representation are translation-invariant, the ergodic theorem implies that, with probability one,
 $$ \frac{\card \{z \in [-n, n]^d : z \ \hbox{is a 1-source} \}}{(2n + 1)^d} \to P (z \ \hbox{is a 1-source}) > 0, $$
 as~$n \to \infty$.
 This implies the lemma.
\end{proof} \vspace*{8pt} \\
 To deduce Theorem~\ref{th:mcp}, fix a~1-source~$z$, which is possible by Lemma~\ref{lem:1-source}. \vspace*{5pt} \\
{\bf Extinction of the~$\pm 2$s}.
 The processes~$\xi$ and~$\eta^{z, 2}$ are coupled in such a way that~$\xi$ has less~$\pm 2$s than~$\eta^{z, 2}$.
 Intuitively, this is true because, on the event~$C_z$, the two processes start with the same configuration of~$\pm 2$s, but the~$\pm 2$s are blocked by the~$\pm 1$s in the process~$\xi$ but not in the process~$\eta^{z, 2}$.
 In addition, because~$p < 1/4d$, it follows from Theorem~\ref{th:extinction} that the process~$\eta^{z, 2}$ dies out, therefore the~$\pm 2$s also die out in the multitype process~$\xi$. \vspace*{5pt} \\
{\bf Survival of the~$\pm 1$s}.
 In general, the~$\pm 1$s in the processes~$\xi$ and~$\bar \eta^{z, 1}$ cannot be compared because~$\xi$ starts with more~$\pm 1$s but the~$\pm 1$s are blocked by the~$\pm 2$s in~$\xi$ but not in~$\bar \eta^{z, 1}$.
 However, because the events~$B_z$ and~$C_z$ occur, the~$\pm 2$s in the process~$\eta^{z, 2}$~(and so the~$\pm 2$s in~$\xi$) don't come in contact with~$\nabla_z$.
 This, together with the attractiveness of the contact process, shows that there are more~$\pm 1$s in the process~$\xi$ than in the process~$\bar \eta^{z, 1}$.
 Finally, because the event~$A_z$ occurs, the process~$\bar \eta^{z, 1}$ survives, therefore the~$\pm 1$s also survive in the multitype process~$\xi$.


\end{document}